\documentclass[a4paper,10pt]{article} 

\usepackage[latin1]{inputenc}
\usepackage[english]{babel}

\usepackage{algorithm}
\usepackage[noend]{algpseudocode}

\usepackage{graphicx}
\usepackage{subfig}
\usepackage{amsmath}
\usepackage{amssymb}
\usepackage{amsthm}
\usepackage{arydshln}
\usepackage{todonotes}
\usepackage{tikz}
\usetikzlibrary{matrix,arrows}
\usepackage{comment}
\newcommand{\R}{{\mathbb{R}}}
\newcommand{\Str}{S}
\newcommand{\ystar}{\mathbf y_*}

\newcommand{\C}{\mathbf c}
\newcommand{\LL}{\mathbf L}
\newcommand{\F}{\mathbf f}
\newcommand{\q}{\hat{\mathbf q}}
\newcommand{\qplain}{{\mathbf q}}
\newcommand{\qt}{\tilde{\mathbf q}}

\newcommand{\Fi}{f}
\newcommand{\Ci}{c}
\newcommand{\Li}{L}

\newcommand{\diffsign}{{\scriptscriptstyle \circ}}
\newcommand{\Fd}{\mathbf f_{\diffsign}}
\newcommand{\Ld}{\mathbf L_{\diffsign}}

\newcommand{\qd}{\mathbf q_{\diffsign}}

\newcommand{\Rd}{R^{\diffsign}_{\scriptscriptstyle x}}
\newcommand{\Sd}[1][c(x)]{S^{\diffsign}_{\scriptscriptstyle #1}}

\newcommand{\Fp}{\mathbf f}
\newcommand{\Lp}{\mathbf L}
\newcommand{\Cp}{\mathbf c}
\newcommand{\rp}{\mathbf r}
\newcommand{\ssp}{\mathbf s}

\newcommand{\Rp}[1][x]{R_{#1}}
\newcommand{\Sp}[1][c(x)]{S_{#1}}

\usepackage{geometry}
\newtheorem{theo}{Theorem}[section]
\newtheorem{mydef}[theo]{Definition}

\newtheorem{prop}[theo]{Proposition}
\newtheorem{lemma}[theo]{Lemma}
\newtheorem{remark}[theo]{Remark}
\newtheorem{assumption}[theo]{Assumption}

\title{An SQP method for equality constrained optimization on manifolds }
\author{Anton Schiela \& Julian Ortiz}

\begin{document}

\maketitle 
\begin{abstract}
We extend the class of SQP methods for equality constrained optimization to the setting of differentiable manifolds. The use of retractions and stratifications allows us to pull back the involved mappings to linear spaces. We study local quadratic convergence to minimizers. 
In addition we present a composite step method for globalization based on cubic regularization of the objective function and affine covariant damped Newton method for feasibility. We show transition to fast local convergence of this scheme.
We test our method on equilibrium problems in finite elasticity where the stable equilibrium position of an inextensible transversely isotropic elastic rod under dead load is sought.

 \vspace{2ex}
  \vspace{2ex}

\end{abstract}

\section{Introduction}

In an important variety of fields, optimization problems benefit from a formulation on nonlinear manifolds. Problems in numerical linear algebra like invariant subspace computations, or low rank 
approximation problems can be tackled using this approach, these problems are the focus of  \cite{absil2009optimization}. Beyond that, plenty of variational problems are posed on infinite dimensional manifolds. 
Nonlinear partial differential equations where the configuration space is 
given by manifolds are found in many applications, for example in liquid crystal physics \cite{prost1995physics} and micro-magnetics \cite{alouges1997new,alouges2012convergent,bartels2007constraint,kritsikis2014beyond,lin1989relaxation}. 
Further are Cosserat materials \cite{badur1989influence} where configurations are maps into the 
space $\mathbb{R}^3\times SO(3)$ which are particularly relevant for shell and rod mechanics.
Similar insights have been successfully
exploited in the analysis of finite strain elasticity and elastoplasticity \cite{ball2002some,mielke2002finite}. 
Further applications of fields with nonlinear codomain are models of topological solitions \cite{manton2004topological}, image processing \cite{tang2000diffusion}, and the
treatment of diffusion-tensor imaging \cite{pennec2006riemannian}. Mathematical literature can be found in \cite{shatah2000geometric} on geometric wave maps, or  \cite{eells1978report} on harmonic
maps. Finally, shape analysis \cite{bauer2010sobolev,ring2012optimization} and shape optimization \cite{schulz2014riemannian} benefits from taking into account the manifold structure of shape spaces. In coupled problems, mixed formulations, or optimal control of the above listed physical models, additional equality constraints occur, and thus one is naturally led to equality constrained optimization on manifolds. 

Unconstrained optimization on manifolds is by now well established, as can be seen in \cite{absil2009optimization,luenberger1972gradient,tang2000diffusion}, where the theory of optimization is covered.  Many things run in parallel to algorithmic approaches on linear spaces. In particular, local (usually quadratic) models are minimized at the current iterate, giving rise to the construction of the 
next step. The main difference between optimization algorithms on a manifold and on linear spaces is how to update the iterates for a given search direction. If the manifold is linear, its tangent
space coincides with the manifold itself and the current iterate can be added to the search direction to obtain the update. If the manifold is nonlinear, the additive update has to be replaced by a 
suitable generalization, a so called \emph{retraction} (cf. e.g. \cite{absil2009optimization}).  Based on these ideas, many algorithms of unconstrained optimization have been carried over to Riemannian manifolds, and 
have been analysed in this framework \cite{huper2004newton,luenberger1972gradient}.
In general, the use of nonlinear retractions enables to exploit given nonlinear problem structure within an optimization algorithm. 

 However, up to now  not much research has been conducted on the construction of algorithms for (equality) constrained optimization on manifolds. A work in the field of shape optimization considers Lagrange-Newton methods on vector bundles \cite{schulz2015towards}. An SQP method for problems on manifolds with inequality constraints was proposed and applied to a problem in robotics, but not analysed in~\cite{Brossette2018}.
 Recently, first order optimality conditions have been derived for equality and inequality constrained optimization on manifolds~\cite{HerzogBergmann2019}. 
  In~\cite{BoumalLiu2019} the extension of known algorithms for constrained optimization to the case where the domain is a manifold has been discussed, which works, if the target space of the constraint mappings is a linear space. The authors consider approaches which allow a reformulation of constrained problems as unconstrained problems on manifolds, such as exact penalty and augmented Lagrangian methods. 
  
\paragraph{Overview.} The main subject of this work is the construction and analysis of SQP methods for equality constrained optimization on manifolds. As our general problem setting we consider smooth Hilbert manifolds $X$ and $Y$ and the problem
\begin{align}\label{eq:mainProblem}
\min_{x\in X}f(x) \,\,\, s.t. \,\,\, c(x)=y_*.
\end{align}
Here $f:X\to \mathbb{R}$ is a twice continuously differentiable functional. The twice continuously differentiable operator $c:X\to Y$  maps from the domain $X$  to the target manifold $Y$, and $y_*\in Y$ is the target point. If $Y$ was a linear space, we could bring $y_*$ to the left hand side and consider the classical constraint $\tilde c(x) := c(x)-y_*=0$. To the best knowledge of the authors, SQP methods have not been considered and analysed in such a setting, so far. 

To define SQP methods for this class of problems we use the popular concept of \emph{retractions}, which map the tangent bundle of a manifold back to the manifold (cf. \cite{absil2009optimization}). As usual, we need a retractions $R:T_x X\to X$ on the domain $X$, but in addition, also the target manifold $Y$ requires, as we will see, a mapping in the other direction $\Str_y :Y \to T_yY$, which we call a \emph{stratification}. 
In this way we pull back both the objective $f$ and the constraint mapping $c$, near a given iterate $x$, to the linear spaces $T_xX$ and $T_{c(x)}Y$. Now, just as in the case of SQP methods on vector spaces, a linear-quadratic model of these pullbacks can be constructed, and a trial correction $\delta x\in T_x X$ can be computed. In Section~\ref{sec:LocalSQP} we will elaborate this approach. In Section~\ref{sec:convergence} we show local quadratic convergence of a general SQP method under suitable assumptions.
Particular emphasis is given to characterize a class of retractions that is non-degenerate close to a local solution. 
Our local theory is based on quanitative assumptions, which can be estimated a-posteriori, and thus provides a theoretical basis for a globalization scheme in the spirit of~\cite{deuflhard2011newton}. 

In the following sections we will consider issues of globalization. Classical globalization concepts are not applicable directly to~\eqref{eq:mainProblem} since evaluation of residuum norms $\|c(x)\|$ is obviously not well defined on a manifold $Y$. A well established algorithmic paradigm that replaces evaluation of residual norms in a natural way is the class of affine covariant Newton methods as discussed in \cite{deuflhard2011newton}. These algorithms dispense with evaluation of norms in the target space and evaluate norms of Newton corrections, instead. 
The concept of affine covariance was carried over from Newton methods to constrained optimization in an \emph{affine covariant composite step method} \cite{lubkoll2017affine} which was used to solve optimal control problems, involving highly nonlinear partial differential equations, such as finite strain elasticity \cite{lubkoll2014optimal}. In this paper we will extend this method to the class of problems, described in \eqref{eq:mainProblem}. 

In this context, we will discuss transition to fast local convergence. 
The subtle issue that arises is, whether the employed globalization strategy can accept full Lagrange-Newton steps, asymptotically. Failure to do so may result in slow local convergence, a behavior that is well known as the \emph{Maratos-effect}. We will study this effect in general for optimization on manifolds and then show that the proposed algorithm does not suffer from this effect. 

Finally, we apply our algorithm to a well known variational problem that is posed on a nonlinear (Hilbert)-manifold: a simple model for an inextensible elastic rod. This serves to illustrate the practical issues of implementation and to demonstrate the feasibility of our approach. 

\section{Local quadratic models and SQP-steps on manifolds}\label{sec:LocalSQP}

In the following, we will consider optimzation problems of the form~\eqref{eq:mainProblem} on smooth Hilbert manifolds $X$ and $Y$ of class $C^2$ (called ``manifolds'' in the following) modelled over Hilbert spaces $\mathbb X$ and $\mathbb Y$. Hilbert manifolds are Banach manifolds, as defined e.g. in \cite[Ch.II]{lang2001fundamentals}, with the special structure that  the domains of their charts are open subsets of Hilbert spaces.  

In the following $T_x X$ denotes the tangent space of $X$ at $x\in X$  
\cite[Ch.II, \S 2]{lang2001fundamentals} which inherits the structure of a Hilbert space from the manifold. 
We denote its zero element $0_x$ and identity mapping $id_{T_x X} : T_xX \to T_x X$. The tangent bundle is denoted by $\pi : TX\to X$ with $T_x X =\pi^{-1}(x)$. We identify $T_{0_x}T_xX\simeq T_xX$, which is possible, because the zero-section of $TX$ can be identified canonically with $X$. Similarly $T_y Y$ is the tangent space of $Y$ at $y\in Y$.

On each tangent space $T_xX$ we introduce a scalar product $\langle \cdot,\cdot,\rangle_x$, which may, but need not necessarily, be defined by a Riemannian metric on $X$, such that $(T_xX,\langle \cdot,\cdot \rangle_x)$ is a Hilbert space, whose topology is compatible with the topology of $X$. By $B^x_\rho:=\{ w\in T_x X : \|w\|_x<\rho\}$ we denote open balls in $T_x X$ around $0_x$. 
We denote the corresponding Riesz-isomorphism by
\begin{align*}
 M_x : T_x X &\to T_x X^*\\
                    v &\mapsto \langle v,\cdot\rangle_x.
\end{align*}
As usual, we have an induced norm $\|v\|_x := \sqrt{\langle v,v\rangle_x}$
and an induced dual norm:
\[
 \|\ell\|_{x,*} := \sup_{v\neq 0} \frac{\ell(v)}{\|v\|_x}=\|M_x^{-1}\ell\|_x. 
\]
Any further structure, like a Riemannian metric or a covariant derivative is not needed for our purpose. 
\subsection{Pull-backs via retractions and stratifications}

Following the approach in~\cite{absil2009optimization} we will introduce the concept of a \emph{retraction} $R_x : T_x X \to X$, which is widely used in unconstrained optimization on manifolds. For the co-domain $Y$ of the constraint mapping $c$ we also  need mappings $\Str_y :Y \to T_y Y$ in the other direction, which we call \emph{stratifications}. On Riemannian manifolds, the exponential map $\exp_x : T_x X \to X$ and the logarithmic map $\log_y : Y\to T_y Y$ are canonical examples.  

\begin{mydef}\label{def:retraction}
 Let $V_{0_x}\subset T_x X$ be a neighbourhood of $0_x \in T_x X$. 
A $C^2$-mapping $R_x : V_{0_x} \to X$ is called a (local) retraction at $x$, if it fulfills:
\begin{itemize}
\item[i)] $R_x(0_x)=x$.
\item[ii)]  $T_{0_x}R_x=id_{T_xX}$.
\end{itemize}
If $V_{x_0}=T_x X$, we call $R_{x}$ globally defined on $T_x X$. 
\vspace{0.2cm}

 Let $U_{y}\subset Y$ be a neighbourhood of $y\in Y$. 
A $C^2$-mapping $\Str_y : U_y \to T_yY$ is called a (local) stratification at $y$, if it fulfills:
\begin{itemize}
\item[i)] $\Str_y(y)=0_y$.
\item[ii)] $T_y \Str_y = id_{T_y Y}$.
\end{itemize}  
If $U_y = Y$, we call $\Str_y$ globally defined on $Y$. 

\end{mydef} 
\begin{remark}
 The original definition of retractions as given e.g. in \cite{absil2009optimization} assumes that each $R_x$ is globally defined. Then $R : TX\to X$ is defined as a smooth map from the tangent bundle to the manifold with the properties, listed above. Be aware that $TR: TTX \to TX$ is already a rather complicated object.  These additional assumptions are, however, not essential for the following and is thus not imposed here. In Section 3 we will discuss regularity assumptions on $R$ with respect to perturbations on $x$ and their relation to smoothness, which will turn out to be sufficient but not necessary for local superlinear convergence. 
\end{remark}
By the inverse mapping theorem  $R_x$ is a local diffeomorphism, 
which means that there is a neighbourhood $V_x^i \subset V_x$, such that
$R_x : V_x^i \to R_x(V_x^i)$ is invertible with continously differentiable inverse $R_x^{-1}$ and:
\[
T_x(R_x^{-1})=(T_{0_x} R_x)^{-1}=id_{T_x X}. 
\]
Thus, locally, the inverse of a retraction $R_x : T_x X \to X$ yields a stratification $\Str_x := R^{-1}_x : U_x \subset X\to T_x X$ and vice versa. 
It also follows that a subset $U_{x}\subset X$ is a neighbourhood of $x$, if and only if there is $\rho >0$, such that $R_{x}(B^{x}_\rho) \subset U_{x}$.

\paragraph{Pullback of a constrained problem.} We define pullbacks of the cost functional $f:X\to \R$ at $x\in X$ via a retraction $R_x$: 
\begin{align*}
\F:T_{x}X&\to \mathbb{R}\\
\F(\delta x)&=(f\circ R_{x})(\delta x)
\end{align*}
and of the constraint $c(x)=y_*$ at $x$ via a retraction $R_x$ and a stratification $\Str_{c(x)}$: 
\begin{align*}
\C:T_{x}X &\to T_{c(x)}Y\\
  \C(\delta x) &:= \Str_{c(x)} \circ c \circ R_{x}(\delta x)-\Str_{c(x)}(y_*).
\end{align*}
In this way, we can now define the pullback of \eqref{eq:mainProblem} to tangent spaces:
\begin{align}\label{eq:pullback_problem}
 \min_{\delta x\in T_x X} \F(\delta x) \quad \mbox{ s.t.}\quad \C(\delta x)=0_{c(x)}, \qquad \C : T_x X \to T_{c(x)} Y.
\end{align}
\begin{prop}
 An element $x_*\in X$ is a local minimizer of~\eqref{eq:mainProblem} if and only if $0_{x_*}$ is a local minimizer of its pullback:
 \begin{equation}\label{eq:pullbackatmin}
  \min_{\delta x \in T_{x_*} X} \F(\delta x)\quad \mbox{ s.t. } \quad \C(\delta x)=0_{c(x_*)}  
 \end{equation}
 \end{prop}
\begin{proof}
 After noting that $c(x_*)=y_*$ is equivalent to $\Str_{c(x_*),i}y_*=0_{c(x_*)}$, the first result follows directly from local invertibility of retractions and stratifications. 
\end{proof}
We can now define a local Lagrangian function via the pullbacks of $f$ and $c$: 
\begin{mydef}
The Lagrangian function of a pullback \eqref{eq:pullback_problem} is given by:
\begin{align}\label{lagrange_par}
\begin{split}
\LL : T_x X \times T_{c(x)}Y^* &\to \R\\
(\delta x,p)&\mapsto \LL(\delta x,p):= \F(\delta x)+{p}\circ \C(\delta x).
\end{split}
\end{align}
\end{mydef}
Thus, we have reduced \eqref{eq:mainProblem} locally to a nonlinear optimization problem on Hilbert spaces, to which techniques of constrained nonlinear optimization on linear spaces can be applied. 
By construction domain and co-domain af $\F$ and $\C$ are Hilbert spaces. Thus, we may to take first and second derivatives of $\F$ and $\C$ in the usual way and obtain bounded linear and bilinear mappings, respectively. Their dependence on the choice of retraction and stratification
will be discussed in the following section.
We will pursue the idea of SQP methods, and thus derive a linearly constrained quadratic model of \eqref{eq:pullback_problem}, which is of the following form:
\begin{align}\label{eq:lin_quad}
 \min_{\delta x\in T_x X} \qplain(\delta x) \quad \mbox{ s.t.}\quad \C'(0_x)\delta x+\C(0_x)=0_y, \qquad \C : T_x X \to T_y Y.
\end{align}
Here $\qplain : T_x X \to \R$ is a quadratic model for 
$\F$, which also will use second order information of the problem. 

\begin{remark}
A different route is taken in \cite{BoumalLiu2019}, where augmented Lagrangian and penalty methods are considered, which transform a constrained problem to an unconstrained problem on manifolds. Then techniques of unconstrained optimization on manifolds are applied.  
\end{remark}

\paragraph{Invariance with respect to retractions and stratifications.}

In differential geometry, quantities which are invariant with respect to changes of charts enjoy particular attention. In this spirit it is natural to ask, which parts the local model \eqref{eq:lin_quad} are invariant against a change of retractions and stratifications. This will turn out to be very useful for the rest of this work. 

Consider a pair $(R_{x,1},R_{x,2})$ of smooth retractions at $x\in X$ and their transition mapping, which is defined on a neighbourhood $V_x$ of $0_x$:
\begin{align*}
 \Phi_x &:= R_{x,1}^{-1}\circ R_{x,2} : V_x \to T_x X.
\end{align*}
Similarly, consider a pair $(\Str_{y,1},\Str_{y,2})$ of smooth stratifications at $y=c(x)\in Y$
and a local transition mapping on a neighbourhood $V_y$ of $0_y$:
\begin{align*}
 \Psi_y &:= \Str_{y,2}\circ \Str_{y,1}^{-1}: V_y \to T_y Y.
\end{align*}
Our first observation is that any two retractions and stratifications coincide up to first order at the origin. Indeed, by the chain rule, we compute:
\begin{align*}
\Phi_x'(0_x)&=(T_{x}R_{x,1}^{-1})\,T_{0_x}R_{x,2}=id_{T_x X},\\
\Psi_y'(0_y)&=(T_{y}\Str_{y,2})\,T_{0_y}\Str_{y,1}^{-1}=id_{T_y Y}.
\end{align*}
As a consequence, there are a couple of quantities, which are invariant against a change of retractions and stratifications.
Defining $\F_i$ and $\C_i$ as the pullbacks of $f$ and $c$ via $R_{x,i}$ and $\Str_{y,i}$, respectively, we  obtain invariance of first derivatives:
\begin{align*}
 \F_2=\F_1\circ \Phi_x \qquad\; \quad &\Rightarrow \quad \F_2'(0_x)=\F_1'(0_x)\Phi_x'(0_x)=\F_1'(0_x)\\
 \C_2=\Psi_y\circ \C_1\circ \Phi_x \quad &\Rightarrow \quad \C_2'(0_x)=\Psi_y'(0_y)\C_1'(0_x)\Phi_x'(0_x)=\C'_1(0_x)
\end{align*}
Due to these observations and to stress invariance, we use non-bold notation:
\begin{align*}
 \Fi(0_x):=\F_1(0_x)=\F_2(0_x),\qquad \qquad \;\; \Fi'(0_x)&:=\F_1'(0_x)=\F_2'(0_x),\\
\Ci'(0_x)&:=\C'_1(0_x)=\C_2'(0_x).
\end{align*}
The definition of $\C_i(0_x)$ involves $\Str_{y,i}(y_*)$. Thus, $\C_i(0_x)$ depends on $\Str_{y,i}$ unless $y= y_*$, but not on $R_{x,i}$. 
To study invariance of the Lagrangian function under changes of retractions, we compute:
\begin{align}\label{change-chart-Lag}
\begin{split}
 \LL_2(\delta x,{p}) &= \F_2(\delta x)+{p}\circ \C_2(\delta x)
 = \F_1\circ \Phi_x(\delta x)+{p}\circ\Psi_y\circ \C_1 \circ \Phi_x(\delta x)\\
 &= \LL_1\circ \Phi_x(\delta x)+{p}\circ (\Psi_y-id_{T_y Y})\circ \C_1 \circ \Phi_x(\delta x).
\end{split}
 \end{align}
Differentiating this identity we obtain invariance of first derivatives (we write $p\Ci'(0_x)$ for the composition $p\circ \Ci'(0_x)$ of linear maps):
 \begin{align*}
 \LL'_2(0_x,{p})=\LL_1'(0_x,{p})=\Fi'(0_x)+p\Ci'(0_x)=:\Li'(0_x,{p}).
\end{align*}
For the second derivative of the Lagrangian, however, we observe a crucial discrepancy:
\begin{lemma}\label{lem:differenceOfLxx}
\begin{equation}\label{eq:differenceOfLxx}
\begin{split}
 (\LL_2''(0_x,{p})&-\LL_1''(0_x,{p}))(v,w)
 =\Li'(0_x,{p})\Phi_x''(0_x)(v,w)+p\,\Psi_y''(0_y)(\Ci'(0_x)v,\Ci'(0_x) w).
 \end{split}
\end{equation}
\end{lemma}
\begin{proof}
We compute by the chain rule and $\Phi_x'(0_x)=id_{T_x X}, \Psi_x'(0_y)=id_{T_y Y}$:
\begin{align}\label{change-chart-2der-Lag}
\begin{split}
 \F_2''(0_x)(v,w) - \F_1''(0_x)(v,w)&=\Fi'(0_x)\Phi_x''(0_x)(v,w)\\
 \C_2''(0_x)(v,w) -\C_1''(0_x)(v,w)
 &=\Psi_y''(0_y)(\Ci'(0_x)v,\Ci'(0_x) w)+\Ci'(0_x)\Phi_x''(0_x)(v,w).
\end{split}
 \end{align}
Adding $\F_i''(0_x)$ and $p\C''_i(0_x)$ yields the desired result. 
 \end{proof} 
We observe that the quadratic model $\q$ in \eqref{eq:lin_quad} may depend on the choice of retractions and stratifications if it involves $\LL''(0_x,p)$. 
 
 \subsection{First and second order optimality conditions}

As a first illustration of the usefulness of invariance considerations, we derive first and second order optimality conditions for our problem on Hilbert manifolds.
For equality and inequality constrained problems on finite dimensional manifolds, where the target spaces of the constraint mappings are linear, \cite{HerzogBergmann2019} have shown first order optimality conditions by different techniques.

 \begin{lemma}\label{lem:StratRet}
 Let $\C : T_x X\to T_{c(x)} Y$ be the pullback of $c:X\to Y$ at $x$ via $R_{x}, \Str_{c(x)}$.
 Assume that $\Ci'(0_x)$ is surjective. Then there is a local retraction $R_{x,0}$, such that
 \begin{equation}\label{eq:stratretr}
   \C_0(\delta x)=\C_0(0_x)+\Ci'(0_x)\delta x  \quad \forall \delta x\in U.
 \end{equation}
\end{lemma}
\begin{proof}
By the surjective implicit function theorem (cf. e.g. \cite[Thm. 4.H]{ZeidlerI}), which is applicable in Hilbert space, there is a neighbourhood $U$ of $0_x$ and a mapping $\phi :U \to T_x X$, such that $\C(\phi(\delta x))=\C(0_x)+\Ci'(0_x)\delta x$ and $\phi'(0_x)=id_{T_x X}$. Defining 
 $R_{x,0} := \phi\circ R_{x}$ we observe by the chain-rule that $R_{x,0}:U\to T_x X$ is indeed a retraction, locally, which satisfies~\eqref{eq:stratretr} 
\end{proof}

\begin{prop}[First order optimality conditions]
 Let $x_*\in X$ be a local minimizer of~\eqref{eq:mainProblem}, and $\F$, $\C$ be any pullback of $f$ and $c$ at $x_*$. Assume that $\Ci'(0_{x_*})$ is surjective. Then there exists a unique Lagrange-Multiplier $p_*\in T_{c(x_*)}Y^*$ such that:
 \begin{equation}\label{eq:KKTconditions}
  \begin{split}
   \Li'(0_{x_*},p_*)&=0^*_{x_*} \quad\quad \mbox{ in } T_{x_*} X^*\\
  \C(0_{x_*})&=0_{c(x_*)} \quad \mbox{ in } T_{c(x_*)} Y.
  \end{split}
 \end{equation}
\end{prop}
\begin{proof}
 Since $\Ci'(0_{x_*})$ is surjective, then \eqref{eq:pullbackatmin} also holds for $R_{x_*,0}$ from Lemma~\ref{lem:StratRet}, so $0_{x_*}$ is the minimizer of
 $\F_0$ on $\ker\,\Ci'(0_{x_*})$. Hence 
 $\Fi'(0_{x_*})v=\F_0'(0_{x_*})v=0$ for all $v\in \ker\,\Ci'(0_{x_*})$,
 so that $-\Fi'(0_{x_*})\in (\ker\,\Ci'(0_{x_*}))^\circ$, the annihilator of $\ker\,\Ci'(0_{x_*})$. Aplication of the  closed range theorem yields $\mathrm{ran}\, \Ci'(0_{x_*})^*=(\ker\,\Ci'(0_{x_*}))^\circ$, so there is an element $p_*\in T_{y_*} Y^*$, such that  $-\Fi'(0_{x_*})=\Ci'(0_{x_*})^*p_*$, or written differently $\Fi'(0_{x_*})+p_* \Ci'(0_{x_*})=0$. 
\end{proof}
 
 \begin{prop}[Second order optimality conditions]
  Assume that $(0_{x_*},p_*)$ satisfy \eqref{eq:KKTconditions} and $\Ci'(0_{x_*})$ is surjective. Then the following implications hold for any pullbacks 
  $\F, \C$ of $f,c$:
  \begin{align}
     \LL''(0_{x_*},p_*)(v,v)&\ge 0 \quad\qquad \forall v\in \ker\,\Ci'(0_{x_*})\quad \Leftarrow \quad x_* \mbox{ is a local minimizer of } \eqref{eq:mainProblem} \label{eq:SNC}    \\    
     \exists \alpha > 0 : \LL''(0_{x_*},p_*)(v,v)&\ge \alpha\|v\|_{x_*}^2 \quad \forall v\in \ker\,\Ci'(0_{x_*})\quad \Rightarrow \quad x_* \mbox{ is a local minimizer of } \eqref{eq:mainProblem}. \label{eq:ellipticSSC}    
  \end{align}
\end{prop}
\begin{proof}
Consider again $R_{x,0}$ from Lemma~\ref{lem:StratRet}, so that the feasible set is $\mathrm{\ker}\,\Ci'(0_{x_*})$. By Taylor expansion we obtain for any $\delta x\in \ker\,\Ci'(0_{x_*})$:
\begin{align*}
 \F_0(\delta x)-\Fi(0_{x_*})&=\LL_0(\delta x,p_*)-\LL(0_{x_*},p_*)=\Li'(0_{x_*},p_*)\delta x+\frac12 \LL_0''(0_{x_*},p_*)(\delta x,\delta x)+o(\|\delta x\|_{x_*}^2)\\
 &= \frac12 \LL_0''(0_{x_*},p_*)(\delta x,\delta x)+o(\|\delta x\|_{x_*}^2).
\end{align*}
Thus, for any $\varepsilon>0$ there is a neighbourhood of $0_{x_*}$, such that
\[
 \frac12 \LL_0''(0_{x_*},p_*)(\delta x,\delta x) -\varepsilon \|\delta x\|_{x_*}^2 \le \F_0(\delta x)-\Fi(0_{x_*})\le \frac12 \LL_0''(0_{x_*},p_*)(\delta x,\delta x) +\varepsilon \|\delta x\|_{x_*}^2. 
\]
Now for $\LL_0''(0_{x_*},p_*)$ \eqref{eq:SNC} follows from the right inequality and \eqref{eq:ellipticSSC} follows from the left inequality. 

Consider an arbitrary $R_x, \Str_{c(x)}$ 
and the the corresponding pullback $\LL(\cdot,p_*)$. By~\eqref{eq:differenceOfLxx} we observe that
\[
 \LL_0''(0_{x_*},p_*)(v,v)= \LL''(0_x,p_*)(v,v) \quad \forall v\in \mathrm{\ker}\,\Ci'(0_{x_*})
\]
since $\Li'(0_{x_*},p_*)=0$ and $\Ci'(0_{x_*})v=0$.
\end{proof}

\subsection{Local quadratic models and the Lagrange-Newton step}

In this section we consider in detail the linear quadratic model~\eqref{eq:lin_quad} of the pullback~\eqref{eq:pullback_problem}. We would like to carry over the ideas of~\cite{absil2009optimization} from unconstrained optimization on Riemannian manifolds to equality constrained optimization on Hilbert manifolds. 

In \cite{absil2009optimization} quadratic models of the objective $f$  are computed independently of the retractions used by the optimization algorithm. First order models use $T_x f$, which coincides with $\Fi'(0_x)$ for any retraction. Second order models are computed by the Riemannian hessian $\mathrm{Hess}\, f$, also known as second covariant derivative. We can view $\mathrm{Hess}\, f$ as a second derivative along geodesics, which is the second derivative $\Fd''(0_x)$ of the pullback $\Fd := f\circ \exp_x$ by the exponential map. This yields a second order quadratic model $\qd$ for $\Fd$. 

If an algorithm is implemented via a retraction $\Rp\neq \Rd$ with pullback $\Fp \neq \Fd$, we see that 
$\Fp''(0_x)\neq \Fd''(0_x)$ in general and thus, $\qd$ is not a second order model for the pullback $\Fp$. In \cite{absil2009optimization} those retractions, for which $\Fd''(0_x)= \Fp''(0_x)$ is guaranteed, are called \emph{second order retractions}. 

From that perspective, steps for unconstrained optimization are computed with the help of \emph{two potentially different retractions}: a natural one $\Rd=\exp_x$, to define a quadratic model $\qd$ and an implemented retraction $\Rp$ to compute an update $\Rp(\delta x)$ from a correction $\delta x\in T_x X$. 

Similarly, in equality constrained optimization, the computation of steps can be split into two parts:
\begin{itemize}
 \item[1.] Optimization algorithms are implemented, using retractions $\Rp$ and  stratifications $\Sp$.  They are needed to evaluate pullbacks $\Fp(\delta x)$ and  $\Cp(\delta x)$ and to compute updates $\Rp(\delta x)\in X$ for $\delta x\in T_xX$.

\item[2.] Linearly constrained quadratic models the nonlinear problem are computed via retractions $\Rd$ and stratifications $\Sd$. If $\Rd=\exp_x$ and $\Sd=\log_{c(x)}$, second covariant derivatives can be used for defining these models. $\Rd$ and  $\Sd$ need not be implemented.
\end{itemize}
As in the unconstrained case, this splitting causes a discrepancy between the pullback and its model. This discrepancy manifests in the transition mappings $\Phi_x$ of $\Rd$, $\Rp$ and $\Psi_{c(x)}$ of $\Sd$,$\Sp$, and in particular in their second derivatives $\Phi_x''$ and $\Psi_{c(x)}''$ via Lemma~\ref{lem:differenceOfLxx} 
\begin{remark}
 This slight shift of perspective, compared to \cite{absil2009optimization} allows us to consider second order methods on manifolds without requiring additional geometric structure (Riemannian metrics, or covariant derivatives) and corresponding advanced concepts of differential geometry on Hilbert manifolds. We hope that this makes our analysis accessible to a wider audience. 
 
 If a Riemannian structure is given, we can cover the purely geometric case, setting $\Rd=\exp_x$ and $\Sd=\log_{c(x)}$. However, we can also cover other cases. For example, if derivatives are computed directly for the implemented pullbacks, we just set $\Rd=\Rp$ and $\Sd=\Sp$.   
\end{remark}
In equality constrained optimization second order quadratic models employ, besides $\F'$ the second derivative of the Lagrangian function, which in our case is $\Ld''(0_x,p_x)$, computed via $\Rd$, $\Sd$. Thus, for some given Lagrange multiplier $p_x\in T_{c(x)}Y^*$ our quadratic model reads:
\begin{align}\label{eq:defq1}
\qd(\delta x):=\Fi(0_{x})+\Fi'(0_{x})\delta x+\frac{1}{2}\Ld''(0_{x},p_x)(\delta x, \delta x).
\end{align}
This leads to the following linearly constrained quadratic optimization problem:
\begin{align}\label{eq:plainQuadraticModel}
 \min_{\delta x\in T_x X} \qd(\delta x) \quad \mbox{ s.t. }\quad  \Ci'(0_x)\delta x+\Cp(0_x)=0.
\end{align}
If a minimizer $\Delta x$ of \eqref{eq:plainQuadraticModel} exists, we call it a \emph{full SQP-step}. 
Adding to $\qd$ any term that is constant on 
$\ker\,\Ci'(0_x)$ does not change $\Delta x$, since all feasible points of \eqref{eq:plainQuadraticModel} differ by an element of $\ker\,\Ci'(0_x)$. So we may add to $\qd$ the term $p_x\Ci'(0_x)$ and thus replace $\Fi'(0_x)\delta x$ by $\Li'(0_x,p_x)\delta x$ in  \eqref{eq:defq1} without changing the minimizer.
Then it follows that a minimizer $\Delta x$ of \eqref{eq:plainQuadraticModel} solves, together with a Lagrange multiplier $\Delta p$, the following system of first order optimality conditions:
\begin{align}\label{eq:Lagrange_Newton}
\left(\begin{array}{cc}
\Ld''(0_x,{p}_x) & \Ci'(0_{x})^{*}\\
\Ci'(0_{x})  & 0
\end{array} \right)
\left(\begin{array}{c}
\Delta x \\
\Delta p
\end{array} \right)
+
\left(\begin{array}{c}
\Li'(0_{x},p_x) \\
\Cp(0_x)
\end{array}\right)
=0.
\end{align}
 We observe, that \eqref{eq:Lagrange_Newton} resembles the Newton system for the first order optimality conditions~\eqref{eq:KKTconditions}, so $\Delta x$ is also called the \emph{Lagrange-Newton step}.

\paragraph{Second order consistency.} Let us specify the case, where $\qd$ is a second order model for $\Fp$ and $\Cp$, which clearly holds, if $\Lp''(0_x,p_x)=\Ld''(0_x,p_x)$. We carry over the definition of second order retractions from~\cite{absil2009optimization} to our setting: 
\begin{mydef}\label{def:pairs}\mbox{}
Pairs of retractions $(\Rd,\Rp)$ and stratifications $(\Sd,\Sp)$, are called second order consistent, if their transition mappings satisfy $\Phi_x''(0_{x})=0$ and $\Psi_y''(0_{y})=0$, respectively.
\end{mydef}
Inserting this definition into of~\eqref{eq:differenceOfLxx} yields:
\begin{prop}
Consider pairs of retractions $(\Rd,\Rp)$ and stratifications $(\Sd,\Sp)$.
\begin{itemize}
 \item[i)] if $(\Rd,\Rp)$ is second order consistent then
 $\Ld''(0_x,{p}_x)\!=\!\Lp''(0_x,{p}_x)$ on $\ker\,\Ci'(0_x)$.
 \item[ii)] if $(\Rd,\Rp)$ and $(\Sd,\Sp)$ are second order consistent, then
 $\Ld''(0_x,{p}_x)=\Lp''(0_x,{p}_x)$ on $T_x X$. 
 \end{itemize}
\end{prop}
It is clear from the definition of second order retrations in \cite{absil2009optimization} that the pair $(\exp_x,\Rp)$ is second order consistent
if and only if $\Rp$ is a second order retraction.

\subsection{Orthogonal splitting of the Lagrange-Newton step} 

Next, we will consider an orthogonal splitting of the Lagrange-Newton step. On the one hand, this will admit a convenient analysis of local convergence of an SQP method, on the other hand this gives us the computational basis for a globalization within the class of \emph{composite step methods}, a class of algorithms has become quite popular in nonlinear optimization
\cite{Vardi1985,Omojokun1990,heinkenschloss2014matrix,lubkoll2017affine}.

 The Lagrange-Newton step $\Delta x$, defined by~\eqref{eq:Lagrange_Newton} is split orthogonally via $\langle\cdot,\cdot\rangle_x$  into a \emph{normal step} 
 $\Delta n \in \ker \Ci'(0_x)^\perp$ and a \emph{tangential step} $\Delta t \in \ker \Ci'(0_x)$, so that $\Delta x=\Delta n+\Delta t$. In other words, we introduce an orthogonal splitting of $T_x X$ into $\ker\,\Ci'(0_x)$ and its orthogonal complement $\ker\,\Ci'(0_x)^\perp$. 
 
 These 
notions relate to the normal and tangential space of the subset $c^{-1}(c(x))\subset X$ at $x$, which is locally a submanifold of $X$, if $\Ci'(0_x)$ is surjective.
The normal step $\Delta n$ can be seen as a Newton step for the underdetermined system $c(x)=y_*$, while the tangential step $\Delta t$ serves as a minimization step. 
For globalization (cf. Section~\ref{sec:global}, below) both components will be modified independently with the aim to achieve progress of the resulting algorithm both in feasibility and optimality. 

\paragraph{Normal step.}
For $y=c(x)$ and $g\in T_{y}Y$, consider the following minimal norm problem: 
\begin{align}\label{eq:minimumnorm}
\min_{w \in T_xX} \frac{1}{2}  \langle  w,w \rangle_x \,\, s.t. \,\, \Ci'(0_{x})w +g=0_y.
\end{align}
This is equivalent to finding $w \in \ker \Ci'(0_{x})^{\perp}$ such that $\Ci'(0_{x})w +g=0_y$. 
If $\mathbf{c'}(0_x)$ is surjective, a feasible solution exists, and by the Lax-Milgram theorem (applied on the Hilbert space $\ker \Ci'(0_{x})$) we obtain a unique optimal solution of \eqref{eq:minimumnorm}. 
Clearly, $w$ (together with a Lagrange multiplier $q$) solves the corresponding first order optimality conditions:
\begin{align}\label{eq:minimumnormKKT}
\left(\begin{array}{cc}
M_x & \Ci'(0_{x})^{*} \\
\Ci'(0_{x}) & 0
\end{array}\right) \left(\begin{array}{c}
w \\
q
\end{array}\right)
+\left(\begin{array}{c}
0 \\
g
\end{array}\right)=0.
\end{align} 
Here $M_x : T_x X \to T_x X^*$ denotes the Riesz-isomorphism and $\Ci'(0_{x})^{*}:T_yY^*\to T_x X^*$ the adjoint of $\Ci'(0_{x})$. 
We write in short:
\[
w=-\Ci'(0_{x})^{-}g
\]
with the linear, bijective operator:
\[
   \Ci'(0_{x})^{-} : T_{c(x)}Y \to \ker\,\Ci'(0_x)^\perp. 
\]
We observe that $\Ci'(0_{x})\Ci'(0_{x})^{-}=id_{T_{y}Y}$ so $\Ci'(0_{x})^-$ is a right pseudo-inverse of $\Ci'(0_{x})$. 
Now we can define the full \emph{normal step} as the solution of \eqref{eq:minimumnormKKT} with $g=\Cp(0_{x})$:
\begin{align}\label{eq:normal_step}
\Delta n:=-\Ci'(0_{x})^{-}\Cp(0_{x}) \in \ker \Ci'(0_x)^\perp
\end{align} 
More generally, if we replace in \eqref{eq:minimumnormKKT} $\Ci'(0_x)$ by (surjective) $\Cp'(v_x)$ for some $v_x\in T_xX$ we obtain in the same way
a mapping
\[
\Cp'(v_x)^-: T_{y}Y \to \ker\,\Cp'(v_x)^\perp.
\]

\paragraph{Lagrange multiplier.} 
To be able to compute $\Ld''(0_x,p_x)$ at $x\in X$ we need to compute Lagrange multiplier estimate $p_x \in T_{y}Y^*$, first. 
A standard way is to define $p_x$ (together with the ``projected gradient'' $v$) as the unique solution of the system:
\begin{align}\label{eq:pKKT}
\left(\begin{array}{cc}
M_x & \Ci'(0_{x})^{*} \\
\Ci'(0_{x}) & 0
\end{array}\right)
\left(\begin{array}{c}
v \\
{p}_x
\end{array}\right)+
\left(\begin{array}{c}
\Fi'(0_x) \\
0
\end{array}\right)=0,
\end{align}
which is the system of first order optimality conditions to the minimization problem:
\[
  \min_{v\in T_x X} \Fi'(0_x)v+\frac12 \langle v,v\rangle_x \quad \mbox{ s.t. } \quad \Ci'(0_x)v =0_y.  
\]
Again, by the Lax-Milgram theorem, this problem has a unique solution $v$, and in $\Ci'(0_x)$ is surjective, $p_x$ is the corresponding unique Lagrange multiplier. 
It can be checked easily that ${p}_x$ satisfies 
\begin{align}\label{Computation_Lag_multiplier}
\Li'(0_x,p_x)w = \Fi'(0_{x}) w+{p}_x\Ci'(0_{x})w=0_x^* \quad \forall w\in \ker \Ci'(0_{x})^{\perp}
\end{align}
and we can write in short:
\[
 p_{x}=-\Fi'(0_x)\Ci'(0_x)^-.
\]
Since $\Li'(0_x,p)v=\Fi'(0_x)v$ for $v\in \ker \Ci'(0_x)$ and any $p\in T_y Y^*$, \eqref{Computation_Lag_multiplier} implies a minimum norm property:
\begin{equation}\label{eq:pminnorm}
 \|\Li'(0_x,p_x)\|_{x,*}\le\|\Li'(0_x,p)\|_{x,*} \quad \forall p\in T_y Y^*.
\end{equation}
Since the data involved in \eqref{eq:pKKT} does not depend on the chosen retraction and stratification, $p_x$ is invariant under a change of retractions and stratifications, but it does depend on the chosen norm. Moreover, if $0_x$ satisifies the first order optimality conditions, then $\Li'(0_x,p_x)=0$. 

More generally, for any $v_x\in T_x X$
\[
 \Fp'(v_x)w+p_{v_x}\Cp'(v_x)w=0_x^* \quad \forall w\in \ker\,\Cp'(v_x)^\perp
\]
is equivalent to 
\begin{equation}\label{eq:pxshort}
 p_{v_x}=-\Fp'(v_x)\Cp'(v_x)^-.
\end{equation}
\paragraph{Tangential step.}
For $\delta x\in T_x X$ consider the unique orthogonal splitting $\delta x:= \delta n + \delta t$ with $\delta n \in \ker \Ci'(0_x)^{\perp}$ and $\delta t \in \ker \Ci'(0_x)$. Then our quadratic model reads:
\begin{align*}
\qd(\delta x)&=\Fi(0_{x})+\Fi'(0_{x})(\delta n +\delta t)+\frac{1}{2}\Ld''(0_{x},p_x)(\delta n+\delta t,\delta n+\delta t). 
\end{align*}
Let us now fix $\delta n$ and consider the problem:
\begin{align}\label{eq:tangentmin}
\min_{\delta t} \qd(\delta n +\delta t) \quad  \mbox{s.t}\quad  \Ci'(0_{x})\delta t =0_y,
\end{align}
which, after adding the term ${p}_{x}\Ci'(0_{x})\delta t=0$ and omitting terms independent of $\delta t$  is equivalent to:
\begin{align*}
\min_{\delta t} &\left( \Li'(0_{x},{p}_x)+\Ld''(0_{x},p_x)\delta n\right)\delta t+\frac{1}{2}\Ld''(0_{x},p_x)(\delta t,\delta t)
\quad \mbox{ s.t.} \quad 
\Ci'(0_x)\delta t=0_y.
\end{align*}
 If $\Ld''(0_x,p_x)$ satisfies \eqref{eq:ellipticSSC} a solution $\Delta t$ (called \emph{tangential step}) of \eqref{eq:tangentmin} exists, using the Lax-Milgram theorem a third time. It solves, together with a Lagrange multiplier $\Delta p$, the following  linear system:
\begin{align}\label{Tangential_STP}
\left(\begin{array}{cc}
\Ld''(0_x,{p}_x) & \Ci'(0_{x})^{*}\\
\Ci'(0_{x})  & 0
\end{array} \right)
\left(\begin{array}{c}
\Delta t \\
\Delta p
\end{array} \right)
+
\left(\begin{array}{c}
\Li'(0_{x},{p}_{x})+ \Ld''(0_x,{p}_x) \delta n \\
0
\end{array}\right)
=0.
\end{align}
In particular, adding the normal step $\Delta n$ and the tangential step $\Delta t$ yields the solution $\Delta x = \Delta n+\Delta t$ of the full Lagrange-Newton system~\eqref{eq:Lagrange_Newton} with the same multiplier $\Delta p$. We recapitulate our findings:
\begin{prop}
 Assume that $\Ci'(0_x)$ is surjective. Then a well defined right pseudo-inverse
 $\Ci'(0_x)^-$ of $\Ci'(0_x)$  exists and $\Delta n=-\Ci'(0_x)^-\Cp(0_x)$ and $p_x=-\Fi'(0_x)\Ci'(0_x)^-$ are uniquely defined. 
 
 If in addition $\Ld''(0_x,p_x)$ satisfies \eqref{eq:ellipticSSC}, then $\Delta t$, and $(\Delta x,\Delta p)$ are uniquely defined as solutions of~\eqref{Tangential_STP} and \eqref{eq:Lagrange_Newton}, respectively.
\end{prop}

For purpose of globalization \emph{composite step methods} compute modified normal steps $\delta n$ and  tangential steps $\delta t$ (using, for example, a line-search, a trust-regions, or cubic regularization), and propose an update $\delta x=\delta n+\delta t$. We will return to this topic in Section~\ref{sec:global}.

\section{A local SQP-Method}\label{sec:convergence}

Let us briefly recapitulate our construction as follows. Starting from a problem on nonlinear spaces:
\[
 \min_{x\in X} f(x) \mbox{ s.t. } c(x)=y_* \mbox{ where  } c : X \to Y
\]
we first \emph{linearize the spaces} via $\Rp, \Sp$ at $x\in X$ to obtain a pullback on tangent spaces:
\[
 \min_{\delta x\in T_xX} \Fp(\delta x) \mbox{ s.t. } \Cp(\delta x)=0_{c(x)} \mbox{ where  } \Cp : T_xX \to T_{c(x)}Y,
\]
then we \emph{linearize the problem} at $0_x$ via $\Rd, \Sd$ to obtain a linearly constrained quadratic problem:
\begin{align*}
 \min_{\delta x\in T_x X} \qd(\delta x) \quad \mbox{ s.t. }\quad  \Ci'(0_x)\delta x+\Cp(0_x)=0_{c(x)}.
\end{align*}
Here $\qd$ is given by~\eqref{eq:defq1} and with $p_x$ from \eqref{eq:pKKT}.  Its solution $\Delta x$ is the full Lagrange-Newton step \eqref{eq:Lagrange_Newton}.

An SQP method creates a sequence of iterates by solving these quadratic problems. The update of iterates is performed by a retraction $x_+ = \Rp(\Delta x)$, which replaces the additive update $x_+=x+\Delta x$ that is used in linear spaces. 

To make an SQP-method well defined, we impose the following assumptions:
\begin{assumption}\label{ass:retractions}
 For each $x\in X$ there are local retractions $\Rd, \Rp$ (with domain of $\Rp$ denoted by $V_x$), local stratifications $\Sd$ and globally defined stratifications $\Sp$.
\end{assumption}
This yields Algorithm~\ref{alg:local _inner_loop_rough}:
\begin{algorithm}[h!]
  \caption{Local SQP method}\label{alg:local _inner_loop_rough}
  \begin{algorithmic}
  \Require initial iterate $x$
  \Repeat 
  \State compute a Lagrange multiplier estimate $p_x \in T_{c(x)}Y^*$ via \eqref{eq:pKKT}
    \State compute $\Delta x$ by solving \eqref{eq:Lagrange_Newton}, using $\Rd$, $\Sd$, $\Sp$
    \State \textbf{if} $\Delta x\not \in V_{x}$ \textbf{then}  terminate: ``update not defined'' 
    \State $x \gets \Rp(\Delta x)$ 
  \Until{ converged }
  \end{algorithmic}
 \end{algorithm}
 
 When it comes to the issue of globalization in Section~\ref{sec:global}, we will deal in a more robust way with the constraint $\Delta x\in V_x$. More generally, also $\Sp$ could be assumed local with $y_*$ in its domain, but to deal with the implicit restrictions, imposed by such a local stratification would be rather cumbersome in practice. In contrast, $\Rd$ and $\Sd$ only have to be defined locally, since they are merely used to define derivatives.  

 Since $p_x$ is updated in terms of the current iterate $x$, this is is an iteration in $x\in X$. The plain Lagrange-Newton method would be an iteration in $(x,p)\in X\times TY^*$ and thus require an initial guess and a vector transport for $p$. 
 It can be seen from \eqref{eq:Lagrange_Newton} that the Lagrange-Newton step $\Delta x$ is independent of the choice of $\Rp$. 
 
\subsection{Non-degeneracy of retractions}\label{sec:localNorm}

While the previous chapter was concerned with computations at specific points, we now have to deal with a sequence of points to analyse convergence of Algorithm~\ref{alg:local _inner_loop_rough}. We have to guarantee that the chosen retractions and norms do not degenerate while approaching a local minimizer $x_*$. Due to the inverse mapping theorem all retractions possess a neighbourhood $V^i_x \subset V_x$ of $0_x$, such that $R_x : V^i_x \to R_x(V^i_x)$ is a diffeomorphism and $R_x(V^i_x) \subset X$ is a neighbourhood of $x$. 


If $x_*\in X$ is some desired solution, and $x\to x_*$, we expect these neighbourhoods to overlap in a larger and larger area and in particular $x_*\in R_x(V^i_x)$, eventually, if retractions are chosen reasonably. However, our purely pointwise assumptions do not guarantee such a behaviour. Up to now, retractions may degenerate close to $x_*$ and $V^i_x$ or $R_x(V^i_x)$ may shrink rapidly, as $x\to x_*$. 

\begin{mydef}
 A family of retractions $R_x$ is called invertible over $M \subset X$, if $M \subset R_x(V^i_x)$ for all $x\in M$. 
\end{mydef}
Indeed, in that case, $R_x^{-1}(\xi)$ is well defined for all $x,\xi \in M$. 
If $M$ is open and $x\in M$, continuity of $R_x$ implies that $R_x^{-1}(M)$ is an open neighbourhood of $0_x$.  

For an invertible family $R_x$ of retractions on $M$ and $x_1,x_2\in M$ we can define a simple nonlinear transport operator from $T_{x_1}X$ to $T_{x_2} X$ as follows:
\begin{align}\label{eq:vectortransport}
 \begin{split}
   \Theta_{x_1\to x_2}:=R_{x_2}^{-1}\circ R_{x_1} : R^{-1}_{x_1}(M) &\to R^{-1}_{x_2}(M).
\end{split}
  \end{align}
 Clearly $\Theta_{x\to x}=id_{T_xX}$ and inverses are given by: 
\[
\Theta_{x_1\to x_2}^{-1}=R_{x_1}^{-1}\circ R_{x_2}=\Theta_{x_2\to x_1}: R^{-1}_{x_2}(M) \to R^{-1}_{x_1}(M).
\] 
If $M$ is open, we obtain a diffeomorphism between neighbourhoods of $0_{x_1}$ and $0_{x_2}$. 
\begin{mydef}
A family $(R_x,\|\cdot\|_x)$ of retractions and norms is called non-degenerate in a neighbourhood $U_{x_*}$ of $x_*\in X$, if the family $R_x$ is invertible over an open set $O \supset U_{x_*}$ and if there is $\rho > 0$  such that $R_x(B_\rho^{x}) \subset O$ for all $x\in U_{x_*}$ and the following Lipschitz condition holds:
 \begin{align}\label{eq:primalNormEquivalence} 
   \exists \gamma : \;\|\Theta_{x_1 \to x_2}(v)-\Theta_{x_1\to x_2}(w)\|_{x_2} \le \gamma\|v-w\|_{x_1} \quad \forall x_1,x_2\in U_{x_*}, \quad \forall v,w\in B_{\rho}^{x_1}.
 \end{align}
\end{mydef}
As a trivial example, let $x_*=0$ in a Hilbert space $(X,\langle \cdot,\cdot\rangle_X)$. For $\rho>0$ consider retractions $R_x(\delta x)=x+\delta x$, with the artifical choice $V_x^i=V_x:=B_{4\rho}^x$ for all $x\in X$. Choosing $O:=B_{2\rho}^0$ and $U_{x_*}=B_\rho^0$ the triangle inequality yields $O\subset R_x(B_{4\rho}^x)$ for all $x\in O$ (and thus invertibility over $O$) and $R_x(B_{\rho}^x) \subset O$ for all $x\in U_{x_*}$. 
Finally, $\Theta_{x_1\to x_2}(v)=v+(x_1-x_2)$ which shows $\gamma=1$. So the family $(R_x, \|\cdot\|_X)$ is non-degenerate on $U_{x_*}$. 
However, if we choose local norms $\|\cdot\|_x\neq \|\cdot\|_X$ very irregularly, even in this example non-degeneracy may be violated. This illustrates that non-degeneracy is also a condition on the local norms. 

\begin{figure}
\centering
 \includegraphics[width=0.95\textwidth]{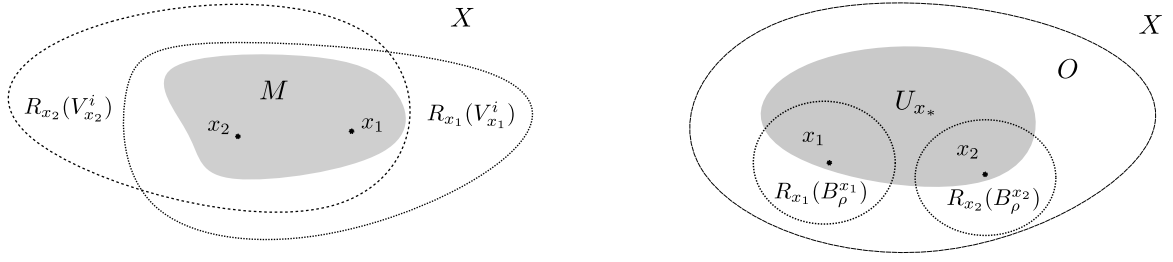}
 \caption{
 Left: $R_x$ is invertible over $M$, if $M \subset R_x(V_x^i)$ for all $x\in M$, so $\Theta_{x_1\to x_2}$ is well defined on $R_{x_1}^{-1}(M)$. Right: Non-degeneracy: $R_x(B_\rho^x) \subset O$ for all $x\in U_{x_*}$, and $\Theta_{x_1\to x_2}$ is Lipschitz on $B^{x_1}_\rho$.
 }
\end{figure}
A necessary condition for non-degeneracy is that in $U_{x_*}$ the domains of definition $V_x$ of $R_x$ all contain a ball of fixed size:
\[
 B_\rho^x \subset R_x^{-1}(O)\subset V_x^i \subset V_x.
\]
Consider $(R_{x},\|\cdot\|_x)$ non-degenerate in a neighbourhood $U_{x_*}$ of $x_*$.
Then $\Theta_{x_1\to x_2}$ is differentiable on $B^{x_1}_\rho$. Hence, by the mean value theorem we obtain the following equivalent condition to~\eqref{eq:primalNormEquivalence} in terms of derivatives:
 \begin{align*}
   \|\Theta'_{x_1\to x_2}(\xi)w\|_{x_2}\le \gamma\|w\|_{x_1} \quad \forall \xi\in B_\rho^{x_1}, \; \forall w\in T_{x_1}X.
 \end{align*}

\paragraph{Dual estimate for pullbacks.} Local non-degeneracy also allows us to establish a dual estimate for derivatives of pullbacks, as we will see in the following. 
Consider a function $f : X \to \R$ and its pullbacks 
\[
 \F := f\circ R_{x} : V_x \to \R \quad \mbox{ and } \quad \F_* := f \circ R_{x_*} : V_{x_*} \to \R,
\]
such that  $\F=\mathbf{f_*} \circ R_{x_*}^{-1}\circ R_{x}=\F_*\circ \Theta_{x\to x_*}$. 
For $x\in U_{x_*}$ we obtain by the chain-rule, setting $\mathbf{x}_* :=\Rp^{-1}(x_*)$:
\[
 \F'(\mathbf x_*)=\F_*'(\Theta_{x\to x_*}(\mathbf x_*))\Theta'_{x\to x_*}(\mathbf x_*)=\Fi'(0_{x_*})\Theta'_{x\to x_*}(\mathbf x_*).
\]
(where $\Fi'(0_{x_*})=\F_*'(0_{x_*})$ is invariant) and thus the dual estimate:
\begin{equation}\label{eq:NormEquivalenceDual}
 \|\F'(\mathbf x_*)\|_{x,*}\le\gamma\|\Fi'(0_{x_*})\|_{x_*,*}.  
\end{equation}
Since the right hand side is independent of $x$, the derivative of the pullback of $f$ at $\mathbf x_*$ is bounded close to $x_*$. 

 \paragraph{Smooth retractions on Riemannian manifolds.} Let us quickly sketch, without going into utmost detail, how smooth retractions $R:TX\to X$ as employed in \cite{absil2009optimization} are non-degenerate. 
 
Assume that $TR : TTX\to TX$ is continuously differentiable and recall that $X$ is modelled over the Hilbert space $(\mathbb X,\langle\cdot,\cdot\rangle_{\mathbb X})$. Let us have a look at the situation in a local chart $\varphi : U \to \mathbb X$, where $U \subset X$. We have the following representation of $R$ with respect to $\varphi$:
 \begin{align*}
 R^\varphi : \varphi(U)\times \mathbb X &\to \mathbb X\\
                     (\mathrm{x},\xi) &\mapsto R^{\varphi}(\mathrm{x},\xi),
 \end{align*}
 such that $R^{\varphi}(\mathrm{x},0)=\mathrm{x}$. We denote the  derivative of 
 $R^\varphi(\mathrm{x},\xi)$ with respect to $\xi$ by $\partial_\xi R^\varphi(\mathrm{x},\xi)$. By assumption $\partial_\xi R^\varphi(\mathrm{x},0) = id_{\mathbb X}$. We would like to study local invertibility of the mapping $R_\mathrm{x}^\varphi : \xi \to R^\varphi(\mathrm{x},\xi)$.

 Since $\partial_\xi R^\varphi$ is continuously differentiable, we may fix $\mathrm{x}_*\in \varphi(U)$ and infer that $\partial_\xi R^\varphi$ is Lipschitz continuous on a neighbourhood $V$ of $(\mathrm{x}_*,0)$. Hence, by the inverse mapping theorem (which yields quantitative results for quantitative assumptions) each $R_\mathrm{x}^\varphi$ is locally invertible on a ball $B_{\rho_{0}}\subset \mathbb X$ around $0$ for $\mathrm{x}\in V$ with $\rho_{0}$ independent of $\mathrm{x}$. In addition $R^\varphi_\mathrm{x}$ and its inverse are Lipschitz continuous with a constant $\sqrt{\gamma_{0}}$ that is also independent of $\mathrm{x}$. 
 
 Arguing as in our trivial example, above, this implies that $R^\varphi$ is invertible over a neighbourhood $O \subset \varphi(U)$ of $\mathrm{x}_*$
 and we also find $V_{\mathrm{x}_*} \subset O$ and $\rho_{\mathbb X}\le \rho_0$ such that $R_\mathrm{x}(B_{\rho_{\mathbb X}}^\mathrm{x}) \subset O$ for all $\mathrm{x}\in V_{\mathrm{x}_*}$.   Moreover, $\Theta^\varphi_{\mathrm{x}_1\to \mathrm{x}_2}=(R^\varphi_{\mathrm{x}_2})^{-1}\circ R^\varphi_{\mathrm{x}_1}$ is defined and Lipschitz continuous with Lipschitz constant $\gamma_{\mathbb X}\le \gamma_0$. Hence, $R^\varphi_\mathrm{x}$, $\|\cdot\|_{\mathbb X}$ are non-degenerate on $V_{\mathrm{x}_*}$.
 Since $\partial_\xi R^\varphi(\mathrm{x},0) = id_{\mathbb X}$ we observe that $\gamma_{\mathbb X}$ can be chosen arbitrarily close to $1$, if $V_{\mathrm{x}_*}$ and $\rho_{\mathbb X}$ are chosen sufficiently small, accordingly.  
 
 A Riemannian metric on $X$ is represented on $\varphi(U)\times \mathbb X$ by a continuous field of scalar products $\langle \cdot,\cdot\rangle_\mathrm{x}$ which are all equivalent to $\langle \cdot,\cdot\rangle_{\mathbb X}:=\langle \cdot,\cdot\rangle_{\mathrm x_*}$. By continuity we find locally uniform constants of equivalence on $V_{\mathrm{x}_*}$ (which can be chosen arbitrarily close to $1$ on correspondingly small choice of $V_{\mathrm{x}_*}$) and thus a uniform radius $\rho$ and Lipschitz constant $\gamma$ can also be established for the Riemannian metric. Still, $\gamma$ can be chosen arbitrarily close to $1$ if $\rho$ and $V_{\mathrm{x}_*}$ are chosen sufficiently small.

\subsection{Local convergence analysis}

 We now study local convergence of Algorithm~\ref{alg:local _inner_loop_rough}.  Compared to standard analysis of Newton's method, the situation is a little more delicate here, because after each step, a different retraction is chosen and a different local norm is used. Furthermore, the question arises, if second order consistent retractions and stratifications are needed for local quadratic convergence of our SQP method. It is already known from Newton methods on manifolds and from unconstrained optimization algorithms~\cite{huper2004newton,absil2009optimization} that local convergence can be achieved for arbitrary retractions.   

We are going to perform our convergence analysis in the framework of \emph{affine covariant Newton methods}. We refer to \cite{deuflhard2011newton} for a detailed account on and motivation of this approach. One consequence for our analysis is that norms on $T_y Y$ do not occur explicitely. 

We will denote by $x_*$ a local solution of \eqref{eq:mainProblem} and capture the nonlinearity of the pullbacks of the problem in the following assumptions:
\begin{assumption}\label{ass:local1}
Assume that $(\Rp,\|\cdot\|_{x})$ are non-degenerate on a neighbourhood $U_{x_*}$ of $x_*$, and that there are constants $\rho_0>0,\omega_{\Fp'},\omega_{\Cp},M_{\Cp}$, independent of $x\in U_{x_*}$, such that the following estimates hold for all $x\in U_{x_*}$, and all $v,\delta x\in T_x X$ with $\|\delta x\|_x\le \rho_0$:
\begin{align}
\label{eq:Lipschitzf}\|\Fp'(\delta x)-\Fi'(0_x)\|_{x,*}  &\leq \omega_{\Fp'} \|\delta x\|_x,\\
\label{eq:Lipschitzc1}
\Ci'(0_x) \mbox{ is surjective, }\quad \|\Ci'(0_x)^{-}(\Cp'(\delta x)-\Ci'(0_x))v\|_x&\leq \omega_{\Cp} \|\delta x\|_x\|v\|_x,\\
\label{eq:boundc}
\mathbf{x}_*:=\Rp^{-1}(x_*),\quad \Cp'(\mathbf x_*) \mbox{ is surjective, }\quad \|\Cp'(\mathbf x_*)^-(\Cp'(\mathbf x_*)-\Ci'(0_x))v\|_x&\le \omega^*_{\Cp}\|\mathbf x_*-0_x\|_x\|v\|_x.
\end{align}
\end{assumption}
While \eqref{eq:Lipschitzf} is a classical Lipschitz condition for $\Fp'$, the quantity $\omega_{\Cp}$ in \eqref{eq:Lipschitzc1} is slightly non-standard. It is called an \emph{affine covariant Lipschitz constant} for $\Cp'$, along the lines of \cite{deuflhard2011newton}. 
In this way no norms on $T_y Y$ need to be specified. Moreover, $\omega_{\Cp}$ can be estimated a-posteriori by an algorithmic parameter $[\omega_{\Cp}]$ that is adapted during the run of a globalized algorithm (see \eqref{eq:estimateomegac} in Section~\ref{sec:global}, below). Finally, \eqref{eq:boundc} excludes that the pullback of the constraints at $x_*$ degenerates, as $x$ approaches $x_*$. 
\begin{remark}
If norms in $T_yY$ are specified, \eqref{eq:Lipschitzc1} and \eqref{eq:boundc} are implied by classical Lipschitz continuity of $\Cp':T_x X\to T_{c(x)} Y$ and boundedness of the operator norm $\|\Cp'(v)^-\|_{T_{c(x)} Y\to T_x X}$. This would make a-priori analysis look more standard, but also render the estimates less sharp, quantitatively, and introduce an artificial dependency on the choice of norms on $T_y Y$. Similarly, one could split these assumptions into conditions on the nonlinearity of $f:X\to \R$ and $c:X\to Y$ and on $\Rp$ and $\Sp$. This, however, would hide the possibility to capture nonlinear problem structure by a juidicious choice of retractions (nonlinear preconditioning) and loosen the close connection between analysis and algorithm to a certain degree.   
\end{remark}

We start with some basic estimates on the derivative of the Lagrangian: 
\begin{lemma}\label{lem:auxLipschitz}
 Let $x_*$ be a local minimizer of \eqref{eq:mainProblem} and suppose that Assumption~\ref{ass:local1} holds near $x_*$. Then there is a neighbourhood $U$ of $x_*$ and constant $\omega_*$, such that for all $x\in U$ and $\mathbf{x}_*:=\Rp^{-1}(x_*)$ the following estimates hold:
 \begin{align}
\label{eq:LLipschitz}  \|\Li'(0_x,p_x)\|_{x,*} &\le \omega_{*}\|0_x-\mathbf{x}_*\|_x,\\
\label{eq:LLipschitz2}  |\Lp'(\mathbf{x}_*,p_x)v| &\le \omega_{*}\omega_{\Cp}\|0_x-\mathbf{x}_*\|^2_x\|v\|_x\quad \forall v\in \ker\, \Ci'(0_x).
 \end{align}
\end{lemma}
\begin{proof}
Choose $U\subset U_{x_*}$ in such a way that $\|0_x-\mathbf x_*\|_x< \rho_0$ for all $x\in U$. 
Since $\mathbf x_*$ is a local minimizer of the pullback~\eqref{eq:pullback_problem} with $\Rp,\Sp$, there is $p_{\mathbf x_*}\in T_{c(x)}Y^*$, such that $\Lp'(\mathbf{x}_*,p_{\mathbf x_*})=0$ with  $p_{\mathbf x_*}=-\Fp'(\mathbf x_*)\Cp'(\mathbf x_*)^-$. This yields for $v\in T_xX$, using $\Ci'(0_x)\Ci'(0_x)^-=id_{T_{c(x)}Y}$:
\begin{align*}
 \Li'(0_x,p_{\mathbf x_*})v&=(\Li'(0_x,p_{\mathbf x_*})-\Lp'(\mathbf x_*,p_{\mathbf x_*}))v=(\Fi'(0_x)-\Fp'(\mathbf{x}_*))v+p_{\mathbf x_*}(\Ci'(0_x)-\Cp'(\mathbf{x}_*))v\\
 &=(\Fi'(0_x)-\Fp'(\mathbf{x}_*))v+\Fp'(\mathbf{x}_*)\Cp'(\mathbf{x}_*)^-(\Cp'(\mathbf{x}_*)-\Ci'(0_x))v.
\end{align*}
This yields via \eqref{eq:Lipschitzf} and \eqref{eq:boundc}:
\begin{equation*}\label{eq:auxLipschitz}
 |\Li'(0_x,p_{\mathbf x_*})v| \le (\omega_{\Fp'}\|0_x-\mathbf{x}_*\|_x+\|\Fp'(\mathbf{x}_*)\|_{x,*}\omega^*_{\Cp}\|0_x-\mathbf{x}_*\|_x)\|v\|_x.
\end{equation*}
By~\eqref{eq:NormEquivalenceDual} $\|\Fp'(\mathbf{x}_*)\|_{x,*}\le\gamma \|\Fi'(0_{x_*})\|_{x_*,*}$ on $U_{x_*}$ and with
$
\omega_* := \omega_{\Fp'}+\omega^*_{\Cp}\gamma\|\Fi'(0_{x_*})\|_{x_*,*}$
we get via~\eqref{eq:pminnorm}:
\begin{equation}\label{eq:auxLipschitz2}
 \|\Li'(0_x,p_x)\|_{x,*}\le \|\Li'(0_x,p_{\mathbf x_*})\|_{x,*} \le \omega_*\|0_x-\mathbf x_*\|_x \quad \forall x\in U.
\end{equation}
Next, we compute from $\Ci'(0_x)\Ci'(0_x)^-=id_{T_{c(x)}Y}$ and $\Li'(0_x,p_x)=0$ on $\ker\,\Ci'(0_x)^\perp=\mathrm{ran}\,\Ci'(0_x)^{-}$:
\begin{align*}
\Lp'(\mathbf x_*,p_x)&=\Lp'(\mathbf x_*,p_x)-\Lp'(\mathbf x_*,p_{\mathbf x_*})=(p_x-p_{\mathbf x_*})\Cp'(\mathbf x_*)=(p_x-p_{\mathbf x_*})\Ci'(0_x)\Ci'(0_x)^-\Cp'(\mathbf x_*)\\
&=\left(\Li'(0_x,p_x)-\Li'(0_x,p_{\mathbf x_*})\right)\Ci'(0_x)^-\Cp'(\mathbf x_*)=-\Li'(0_x,p_{\mathbf x_*})\Ci'(0_x)^-\Cp'(\mathbf x_*).  
\end{align*}
Thus, for $v\in \ker\,\Ci'(0_x)$ we get by \eqref{eq:auxLipschitz2} and \eqref{eq:Lipschitzc1} (with $\delta x=\mathbf{x_*}$):
\begin{align*}
 |\Lp'(\mathbf{x}_*,p_x)v|&=|\Li'(0_x,p_{\mathbf x_*})\Ci'(0_x)^-(\Cp'(\mathbf x_*)-\Ci'(0_x))v|\le \|\Li'(0_x,p_{\mathbf x_*})\|_{x,*}\|\Ci'(0_x)^{-}(\Cp'(\mathbf{x}_*)-\Ci'(0_x))v\|_x\\
 &\le \omega_*\|\mathbf{x}_*-0_x\|_x\omega_{\Cp}\|\mathbf{x}_*-0_x\|_x\|v\|_x.
\end{align*}
\end{proof}
Next we show for the pullback of our problem at $T_xX$, where $x$ is the current iterate, that one Newton step reduces the error quadratically. We need the following additional assumption:
\begin{assumption}\label{ass:local2}
Let $U_{x_*}$ be a neighbourhood of $x_*$ and assume that $(\Rp,\|\cdot\|_x)$ are non-degenerate on $U_{x_*}$.  Assume further that there are constants $\rho_0>0,\omega_{\Lp}$, $\alpha_{\Ld''}>0,M_{\Ld''}$, $M_\phi$ such that for all $x\in U_{x_*}$:
\begin{align}
\label{eq:LipschitzL} \vert (\Lp''(\delta x,p_x)-\Lp''(0_x,p_x))(v,w) \vert &\leq \omega_{\Lp} \|\delta x\|_x\|v\|_x\|w\|_x, \quad \forall \delta x,v,w\in T_x X: \|\delta x\|_x\le \rho_0\\
 \label{eq:ellipticity} \Ld''(0_x,p_x)(v,v)&\ge \alpha_{\Ld''}\|v\|_x^2 \qquad \qquad \qquad \forall v\in \ker \Ci'(0_x)\\
 \label{eq:boundedLpp} |\Ld''(0_x,p_x)(v,w)|&\le M_{\Ld''}\|v\|_x\|w\|_x \qquad \quad \;\;\forall v,w\in T_x X\\
 \label{eq:boundedRetraction} \|\Phi_x''(0_x)(v,w)\|_x &\le M_\Phi \|v\|_x\|w\|_x \qquad \qquad \forall v,w\in T_x X
\end{align}
\end{assumption}
The first three assumptions are all fairly standard for local convergence analysis: smoothness, regularity, and boundedness of the second order term. Compared to the setting of vector spaces, only \eqref{eq:boundedRetraction}, a uniform bound on the second derivative of the transition mapping $\Phi_x$ of $\Rd,\Rp$ is new. No explicit consistency assumptions on the stratifications are needed. 
\begin{prop}\label{pro:localNewton}
 Suppose that Assumption~\ref{ass:local1} and Assumption~\ref{ass:local2} hold at a local minimizer $x_*$. Then, there is
 a neighbourhood $U$ of $x_*$, such that the following holds for all $x\in U$:
 
 There is a constant $\beta$, independent of $x \in U$, such that with $\mathbf{x}_*:=\Rp^{-1}(x_*)$ we obtain on $T_x X$:
 \begin{equation}\label{eq:newtonquad}
   \|0_x+\Delta x-\mathbf{x}_*\|_x \le \beta\|0_x-\mathbf{x}_*\|^2_x.   
 \end{equation}
 \end{prop}
\begin{proof}
In the following analysis we will only need to apply the estimates of Assumption~\ref{ass:local1} and~\ref{ass:local2} for $\delta x$ being a convex combination of $0_x$ and $\mathbf x_*$ in $T_x X$. We thus choose the neighbourhood $U \subset\Rp[x_*](B^{x_*}_{\rho_*})$ as the image of some ball $B^{x_*}_{\rho_*}$ via $\Rp[x_*]$, intersected by the neighbourhood of $x_*$, in which Assumption~\ref{ass:local1} and~\ref{ass:local2} hold. Specifying $\rho_* \le \gamma^{-1}\rho_0$ implies $\|0_x-\mathbf x_*\|_x\le \rho_0$ by \eqref{eq:primalNormEquivalence}, which is sufficient to justify the application of  Assumption~\ref{ass:local1} and~\ref{ass:local2} and the results of Lemma~\ref{lem:auxLipschitz}. 

Adding in the Lagrange-Newton step \eqref{eq:Lagrange_Newton} the vector $0_x-\mathbf x_*$ to $\Delta x$, compensating this in the right hand side, and subtracting $\Cp(\mathbf x_*)=0_y$, we obtain  the identity:
\begin{align}\label{eq:NewtonErrorSystem}
\left(\begin{array}{cc}
\Ld''(0_x,{p}_x) & \Ci'(0_{x})^{*}\\
\Ci'(0_{x})  & 0
\end{array} \right)
\left(\begin{array}{c}
0_x+\Delta x-\mathbf x_* \\
\Delta p
\end{array} \right)
+
\left(\begin{array}{c}
\Li'(0_{x},p_x)-\Ld''(0_x,{p}_x)(0_x-\mathbf x_*) \\
\C(0_x)-\Cp(\mathbf x_*)-\Ci'(0_x)(0_x-\mathbf x_*)
\end{array}\right)
=0.
\end{align}
To show \eqref{eq:newtonquad} we will analyse~\eqref{eq:NewtonErrorSystem} by the following orthogonal splitting:
 \[
 0_x+\Delta x-\mathbf{x}_*=n_*+t_*, \quad n_* \in \ker\, \Ci'(0_x)^\perp, \quad t_*\in \ker\, \Ci'(0_x)
 \]
 and estimate $\|n_*\|_x$ and $\|t_*\|_x$, separately. 
 
 For the normal part $n_*$, we use \eqref{eq:minimumnormKKT}
 with $g=\C(0_x)-\Cp(\mathbf x_*)-\Ci'(0_x)(0_x-\mathbf x_*)$:
\[
 n_* = \Ci'(0_x)^-(\Ci'(0_x)(0_x-\mathbf{x}_*)-(\C(0_x)-\Cp(\mathbf{x}_*))),
\]
which implies $n_* \in \ker\, \Ci'(0_x)^\perp$ satisfies the second row of \eqref{eq:NewtonErrorSystem}.

Application of \eqref{eq:Lipschitzc1} yields via the fundamental theorem of calculus and \eqref{eq:Lipschitzc1} (with $\delta x=(1-s)\mathbf x_*$, $s\in[0,1])$:
\begin{align*}
\|n_*\|_x &= \left\|\int_0^1\Ci'(0_x)^-(\Ci'(0_x)-\Cp'((1-s)\mathbf{x}_*)(0_x-\mathbf{x}_*)\,ds\right\|_x\\
&\le 
\int_0^1 \omega_{\Cp} (1-s)\|0_x-\mathbf x_*\|_x^2\, ds=
\frac{\omega_{\Cp}}{2}\|0_x-\mathbf{x}_*\|_x^2. 
\end{align*}
Since $\Ci'(0_x)t_*=0$, $n_*+t_*$ satisfies the second row of \eqref{eq:NewtonErrorSystem}.  The remaining first row of \eqref{eq:NewtonErrorSystem} then reads:
\begin{align*}
\Ld''(0_x,{p}_x)(t_*+n_*) +\Ci'(0_{x})^{*}\Delta p&= \Ld''(0_x,p_x)(0_x-\mathbf{x}_*)-\Li'(0_x,p_x).
\end{align*}
This is an equation in $T_xX^*$. Testing with $v\in \ker\,\Ci'(0_x)$ and computing $(\Ci'(0_{x})^{*}\Delta p)v=\Delta p\,\Ci'(0_{x})v=0$ we obtain:
\begin{align}\label{eq:tplus}
\Ld''(0_x,{p}_x)(t_*,v)= \Ld''(0_x,p_x)(0_x-\mathbf{x}_*,v)-\Li'(0_x,p_x)v-\Ld''(0_x,{p}_x)(n_*,v)\quad \forall v\in \ker\,\Ci'(0_x).
\end{align}
Next, we  derive an estimate for the right hand side of the form:
\begin{align}\label{eq:desiredlest}
 |(\Ld''(0_x,p_x)(0_x-\mathbf{x}_*)-\Li'(0_x,p_x)-\Ld''(0_x,p_x)n_*)v|
 \le M_*\|0_x-\mathbf{x}_*\|_x^2\|v\|_x \quad \forall v\in\ker\, \Ci'(0_x)
\end{align}
to obtain a suitable bound for $\|t_*\|_x$ via ellipticity~\eqref{eq:ellipticity}.
To show \eqref{eq:desiredlest} we first observe
\[
|\Ld''(0_x,p_x)(n_*,v)| \le M_{\Ld''}\|n_*\|_x\|v\|_x\le M_{\Ld''}\frac{\omega_{\Cp}}{2}\|0_x-\mathbf{x}_*\|_x^2\|v\|_x.
\]
Next we telescope:
\begin{align*}
 (\Ld''(0_x,p_x)(0_x-\mathbf{x}_*)-\Li'(0_x,p_x))v
 &=(\Ld''(0_x,p_x)-\Lp''(0_x,p_x))(0_x-\mathbf{x}_*,v)\\
 &+(\Lp''(0_x,p_x)(0_x-\mathbf{x}_*)-\Li'(0_x,p_x))v
\end{align*}
into a sum of two terms. The first term (which vanishes if $\Rd$ and $\Rp$ are second order consistent) is estimated via \eqref{eq:differenceOfLxx} and  \eqref{eq:boundedRetraction}: \begin{align*}
 |(\Ld''(0_x,p_x)&-\Lp''(0_x,p_x))(0_x-\mathbf{x}_*,v)|\\
 &=|\Li'(0_x,p_x)\Phi_{x}''(0_x)(0_x-\mathbf{x}_*,v)+p_x\Psi_y''(0_{c(x)})(\Ci'(0_x)(0_x-\mathbf{x}_*),\Ci'(0_x) v)|\\
 &\le \|\Li'(0_x,p_x)\|_{x,*}M_{\Phi}\|0_x-\mathbf{x}_*\|_x\|v\|_x
 \le \omega_*M_{\Phi}\|0_x-\mathbf{x}_*\|^2_x\|v\|_x.
\end{align*}
Observe that $\Psi''_y(0_{c(x)})$ dropped out, because $v\in \ker\, \Ci'(0_x)$.

The second term is estimated via~\eqref{eq:LipschitzL}, using again the fundamental theorem of calculus, and~\eqref{eq:LLipschitz2} due to $v\in \ker\,\Ci'(0_x)$:
\begin{align*}
|(\Lp''(0_x,p_x)&(0_x-\mathbf{x}_*)-\Li'(0_x,p_x))v|\\
&\le
 |\left(\Lp''(0_x,p_x)(0_x-\mathbf{x}_*)-(\Li'(0_x,p_x)-\Lp'(\mathbf{x}_*,p_x))\right)v|+|\Lp'(\mathbf{x}_*,p_x)v|\\
 &=\left|\int_0^1 \left(\Lp''(0_x,p_x)-\Lp''((1-s)\mathbf{x}_*,p_x)\right)(0_x-\mathbf x_*,v)\,ds\right|+|\Lp'(\mathbf{x}_*,p_x)v|\\
 &\le \frac{\omega_{\Lp}}{2}\|0_x-\mathbf{x}_*\|_x^2\|v\|_x+\omega_*\omega_{\Cp}\|0_x-\mathbf{x}_*\|_x^2\|v\|_x. 
\end{align*}
Adding all these estimates yields \eqref{eq:desiredlest} with
\[
 M_*=M_{\Ld''}\frac{\omega_{\Cp}}{2}+\omega_*M_\Phi+\frac{\omega_{\Lp}}{2}+\omega_*\omega_{\Cp}. 
\]
Thus, inserting $v=t_*$ into \eqref{eq:tplus}, using ellipticity~\eqref{eq:ellipticity} and \eqref{eq:desiredlest} we obtain
\[
 \|t_*\|_x \le \frac{M_*}{\alpha_{\Ld''}}\|0_x-\mathbf{x}_*\|_x^2.
\]
In total we obtain by the triangle inequality:
\[
 \|0_x+\Delta x-\mathbf{x}_*\|_x\le \|n_*\|_x+\|t_*\|_x\le \left(\frac{\omega_{\Cp}}{2}+\frac{M_*}{\alpha_{\Ld''}}\right)\|0_x-\mathbf{x}_*\|_x^2
\]
 and thus the desired result. 
\end{proof}

Finally, we exploit non-degeneracy to show that quadratic convergence can be observed for the pullback of the SQP-sequence to $T_{x_*} X$. For that we will need the local vector transport $\Theta_{x_*\to x}$ defined in \eqref{eq:vectortransport}.
\begin{theo}\label{thm:localNewton}
Suppose that Assumption~\ref{ass:local1} and Assumption~\ref{ass:local2} hold at a local minimizer $x_*$. Assume that $x_0$ is sufficiently close $x_*$, i.e., that $\|\Rp[x_*]^{-1}(x_0)-0_{x_*}\|_{x_*}$ is sufficiently small. 
 
 Then Algorithm~\ref{alg:local _inner_loop_rough} creates a sequence $x_k \in X$, defined by $x_{k+1}=\Rp[x_k](\Delta x_k)$ that converges quadratically towards $x_*$.
 This means, if we denote $\mathbf{x}_k := \Rp[x_*]^{-1}(x_k)$ there is $\beta_*$, such that:
 \begin{equation}\label{eq:quadraticconvergence}
   \|\mathbf{x}_{k+1}-0_{x_*}\|_{x_*}\le \beta_*\|\mathbf{x}_{k}-0_{x_*}\|^2_{x_*}.
 \end{equation}
\end{theo}
\begin{proof}
 Consider $x_k$ in the neighbourhood $U$ of $x_*$, used in Proposition~\ref{pro:localNewton}.
We use non-degeneracy~\eqref{eq:primalNormEquivalence}, and Proposition~\ref{pro:localNewton}, denoting $\mathbf{x}_{*,k}=\Rp[x_k]^{-1}(x_*)$:
\begin{align*}
   \|\mathbf{x}_{k+1}-0_{x_*}\|_{x_*}&=\|\Theta_{x_k\to x_*}(0_{x_k}+\Delta x_k)-\Theta_{x_k\to x_*}(\mathbf x_{*,k})\|_{x_k}\\ 
   &=\gamma\|0_{x_k}+\Delta x_k-\mathbf{x}_{*,k}\|_{x_k} \le \beta\gamma\|0_{x_k}-\mathbf{x}_{*,k}\|^2_{x_k}=\beta\gamma\|\Theta_{x_*\to x_k}(\mathbf x_k)-\Theta_{x_*\to x_k}(0_{x_*})\|^2_{x_k} \\
   &\le\beta\gamma^3\|\mathbf{x}_{k}-0_{x_*}\|_{x_*}^2,
 \end{align*}
 so \eqref{eq:quadraticconvergence} holds with $\beta_*=\beta\gamma^3$. 
 $\Rp[x_*]^{-1}(U)$ is a neighbourhood of $0_{x_*}$ and thus contains a ball $B^{x_*}_{\rho_1}$. Choosing $\mathbf x_0$ in a ball $B^{x_*}_{\rho_2}$ of radius $\rho_2 := \min\{\rho_1,1/(2\beta_*)\}$ 
implies by induction that 
\[
  \|\mathbf x_{k+1}-0_{x_*}\|_{x_*}\le \frac12\|\mathbf x_{k}-0_{x_*}\|_{x_*}< \frac{\rho_2}{2^{k+1}}.
  \]
  So $\mathbf x_k \in B^{x_*}_{\rho_2}$ for all $k$ and $\mathbf x_k \to 0_{x_*}$
   with \eqref{eq:quadraticconvergence}.
\end{proof}
We conclude that second order consistency of retractions and stratifications is not needed for local quadratic convergence of Algorithm~\ref{alg:local _inner_loop_rough}.

\section{Consistency of Quadratic Models and the Maratos Effect}

To obtain a robust optimization algorithm, SQP methods have to be equipped with a globalization strategy. Roughly speaking, any such strategy computes modified trial corrections $\delta x$ instead of full steps $\Delta x$ and performs an acceptance test to decide whether $\delta x$ can be used as an update. Typically, such a test involves, among other things, a comparison between the objective $f$ (or a merit function) and the quadratic model $q$, used by SQP. For details on this classical subject we refer to the standard literature on nonlinear optimization (e.g. \cite{conn2000trust,Nocedal1999}).

For second order methods a globalization scheme should support fast local convergence, which means that eventually or at least asymptotically full SQP steps $\Delta x$ are used. It turns out that one necessary ingredient for this is a relation of the form $f(x+\Delta x)-q(\Delta x)= o(\|\Delta x\|^2)$. Unfortunately, in equality constrained optimization this relation is \emph{false} for second order methods, because $q$ involves $L''(x,p)$ and not $f''(x)$ as a second order term. We only have $f(x+\Delta x)-q(\Delta x)= O(\|\Delta x\|^2)$, a difficulty that is known as the root of the \emph{Maratos effect}. If an algorithm suffers from the Maratos effect (attributed to \cite{maratos1978}), its globalization scheme sometimes rejects full Lagrange-Newton steps, even if iterates are already very close to the optimal solution. This may result in an undesirable slow-down of local convergence: only linear, but not superlinear convergence is observed, occasionally. 

Various modifications have been proposed to overcome this problem. A popular variant is to apply a so called \emph{second order correction} $\delta s$, such that $f(x+\Delta x+\delta s)-q(\Delta x)= o(\|\Delta x\|^2)$ can be shown. A thorough discussion of this classical issue in constrained optimizaton and its algorithmic implications can be found e.g. in \cite[Ch. 15.5/6]{Nocedal1999}.

Clearly, similar issues will arise in the more general framework of equality constrained optimization on manifolds. In addition, retractions and stratifications and their consistency may play a role for the order of the error between objective functional and its quadratic model. The purpose of this section is to discuss these issues. 

\subsection{Standard quadratic model}

Let us recall our setting of Section~\ref{sec:LocalSQP} at a given point $x\in X$ and $y=c(x)$. In particular, we use pairs of retractions $(\Rd,\Rp)$, stratifications $(\Sd[y],\Sp[y])$ with transition mappings $\Phi_x$, $\Psi_{y}$, pullbacks $\F$, $\C$, and their derivatives. Our quadratic model was defined in \eqref{eq:defq1} by:
\[
 \qd(\delta x):= \Fi(0_x)+\Fi'(0_x)\delta x+\frac12\Ld''(0_x,p_x)(\delta x,\delta x).
\]
For arbitrary $\delta x\in T_x X$ we consider the orthogonal splitting $\delta x=\delta n+\delta t$ into a normal component $\delta n\in \ker\,\Ci'(0_{x})^\perp$ and a tangential component $\delta t\in \ker\,\Ci'(0_{x})$. 
\paragraph{Simplified normal step/second order correction.}
Let $\delta x\in T_x X$ be arbitrary. 
We compute a simplified normal step, also called a second order correction as follows:
\begin{align}\label{simplifiednormstepcharts}
\delta s:=-\Ci'(0_{x})^{-}\left( \Cp(\delta x)- \Cp(0_{x})-\Ci'(0_{x})\delta x \right)\in T_x X.
\end{align}
Its purpose is twofold:
\begin{itemize}
 \item[i)] Following the ideas of~\cite{deuflhard2011newton} this step can be used to construct an affine invariant globalization mechanism with respect to feasibility that does not require an evaluation of $\|c(x)\|$. We will elaborate more on this in Section~\ref{sec:global}, below. 
 \item[ii)] This step serves as a second order correction, used to counter-act the Maratos effect. 
\end{itemize}
For a trial correction $\delta x$, new iterates are now computed using $\Rp$, namely:
\[
 x_+ := \Rp(\delta x+\delta s).
\]
Thus, for the new objective function value, we obtain:
\[
 f(x_+) =f(\Rp(\delta x+\delta s)) = \Fp(\delta x+\delta s). 
\]
Following the introductory discussion concerning the Maratos effect, we are going to analyse the error $|\Fp(\delta x+\delta s)-\qd(\delta x)|$ between the functional and its quadratic model. It would be desirabe to have an error of $o(\|\delta x\|_x^2)$. However, we have to take into account that the pairs $(\Rd,\Rp)$ and $(\Sd[y],\Sp[y])$ also will introduce an additional error: 
\begin{lemma}\label{ModelEstimate}
The following identity holds:
\begin{align*}
\begin{split}
\Fp(\delta x& + \delta s)-\qd(\delta x)=\rp(\delta x)+ \ssp(\delta x) + \frac12\left(\Li'(0_x,p_x)\Phi_x''(0_x)(\delta x,\delta x)+p_x\Psi_y''(0_y)(\Ci'(0_x)\delta n,\Ci'(0_x)\delta n)\right)
 \end{split}
\end{align*}
with higher order terms $\rp(\delta x)$ and $\ssp(\delta x)$, given by:
\begin{align*}
\rp(\delta x)&:=\Lp(\delta x,p_x)-\LL(0_x,p_x)-\Li'(0_x,p_x)\delta x-\frac{1}{2}\Lp''(0_x,p_x)(\delta x,\delta x) \\
\ssp(\delta x)&:=\Fp(\delta x+ \delta s)-\Fp(\delta x)-\Fi'(0_x)\delta s.
\end{align*}
\end{lemma}
\begin{proof}
We compute
\begin{align*}
\rp(\delta x)&+\qd(\delta x)
= \Lp(\delta x,p_x)-\LL(0_x,p_x)-\Li'(0_x,p_x)\delta x-\frac12 \Lp''(0_x,p_x)(\delta x,\delta x) \\
&+ \Fi(0_x) + \Fi'(0_x,p_x)\delta x+ \frac12 \Ld''(0_x,p_x)(\delta x,\delta x) \\
&=\Fp(\delta x) + p_x[\Cp(\delta x)-\C(0_x)-\Ci'(0_x)\delta x]+\frac12\left(\Ld''(0_x,p_x)-\Lp''(0_x,p_x)\right)(\delta x,\delta x)\\
&=\Fp(\delta x)-p_x\Ci'(0_x)\delta s
-\frac12\left(\Li'(0_x,p_x)\Phi_x''(0_x)(\delta x,\delta x)+p_x \Psi_y''(0_y)(\Ci'(0_x)\delta x,\Ci'(0_x)\delta x)\right),  
\end{align*}
where the identity  $(\ref{eq:differenceOfLxx})$ has been used. Given that $\Fi'(0_x)\delta s=-p_x\Ci'(0_x)\delta s$ and adding and subtracting $\Fp(\delta x + \delta s)$, we obtain
\begin{align*}
 \rp(\delta x)+\qd(\delta x)
 &=\Fp(\delta x+\delta s) - (\Fp(\delta x + \delta s)-\Fp(\delta x )-\Fi'(0_x)\delta s)\\
 &- \frac12\left(\Li'(0_x,p_x)\Phi_x''(0_x)(\delta x,\delta x)+ p_x \Psi_y''(0_y)(\Ci'(0_x)\delta x,\Ci'(0_x)\delta x\right).
\end{align*}
Introducing $\ssp(\delta x)$ and observing $\Ci'(0_x)\delta x=\Ci'(0_x)\delta n$ we obtain the desired result.
\end{proof}

Next, we quantify the size of the higher order terms $\rp(\delta x)$ and $\ssp(\delta x)$, using the assumptions from Section~\ref{sec:convergence} needed for local fast convergence.

\begin{lemma}\label{prop:LipschitzConsistency}
 Assume that \eqref{eq:Lipschitzf}, \eqref{eq:Lipschitzc1}, and $\eqref{eq:LipschitzL}$ hold at $x\in X$ with constants $\omega_{\Fp'}$, $\omega_{\Cp}$ and $\omega_{\Lp}$. Then we conclude for sufficiently small $\delta x$:
 \begin{align}
 \label{eq:estds}  \|\delta s\|_{x} &\le \frac{\omega_{\Cp}}{2}\|\delta x\|_x^2=O(\|\delta x\|_x^2),\\
  \rp(\delta x) &\le \frac{\omega_{\Lp}}{6}\|\delta x\|_x^3=O(\|\delta x\|_x^3),\\
  \ssp(\delta x) &\le \frac{\omega_{\Fp'}}{2}\|\delta s\|_x\left(2\|\delta x\|_x+\|\delta s\|_x \right)=O(\|\delta x\|_x^3).
 \end{align}
 Thus, in particular
\begin{align*}
\begin{split}
\Fp(\delta x& + \delta s)-\qd(\delta x)=\frac12\left(\Li'(0_x,p_x)\Phi_x''(0_x)(\delta x,\delta x)+p_x\Psi_y''(0_y)(\Ci'(0_x)\delta n,\Ci'(0_x)\delta n)\right)+O(\|\delta x\|_x^3).
 \end{split}
\end{align*}
\end{lemma}
\begin{proof}
The results are straightforward applications of the fundamental theorem of calculus.  
Setting $v=\sigma \delta x$ in \eqref{eq:Lipschitzc1}, we have that 
\begin{align*}
\|\delta s\|_x&\leq \int_0^1 \frac{1}{\sigma}\|\Ci'(0_x)^{-}(\Cp'(\sigma \delta x)-\Ci'(0_x))\sigma \delta x\|_x \,d\sigma  
\leq \frac{\omega_{\Cp}}{2} \|\delta x\|_x^2 
\end{align*}
By \eqref{eq:LipschitzL}, we get by applying the fundamental theorem of calculus two times:
\begin{align*}
\vert \rp(\delta x) \vert&\leq \int_0^1 \int_0^1 \frac{1}{\tau^2 \sigma}\vert (\Lp''(\tau \sigma  \delta x,p_x)-\Lp''(0_x,p_x))(\tau \sigma \delta x,\tau \sigma \delta x) \vert \,d\tau d\sigma\\ 
&\leq \omega_{\Lp}\|\delta x\|_x^3 \int_0^1\int_0^1 \tau \sigma^2 \, d \tau d \sigma =\frac{\omega_{\Lp}}{6}\|\delta x\|_x^3.
\end{align*}
and for $\ssp$ we obtain similarly by \eqref{eq:Lipschitzf}:
 \begin{align*}
\vert \ssp(\delta x)\vert &\leq \int_0^1 \vert (\Fp'(\delta x+\sigma \delta s)-\Fi'(0_x))\delta s \vert d \sigma \leq \omega_{\Fp'}\|\delta s\|_x\int_0^1 \|\delta x+ \sigma \delta s\|_x \, d \sigma\\
&\leq \omega_{\Fp'}\|\delta s\|_x\left(\|\delta x\|_x+\frac{1}{2}\|\delta s\|_x \right)\le \frac{\omega_{\Fp'}\omega_{\Cp}}{2}\|\delta x\|_x^2\left(\|\delta x\|_x+\frac{\omega_{\Cp}}{4}\|\delta x\|_x^2 \right).
\end{align*} 
\end{proof}

\paragraph{Discussion.}

Obviously, $|\Fp(\delta x+\delta s)-\qd(\delta x)|=o(\|\delta x\|_2)$, if the employed retractions and stratifications are both \emph{second order consistent}, since then $\Phi_x''(0_x)$ and $\Psi_y''(0_y)$ vanish by definition. 

However, if retractions and stratifications are not second order consistent, we observe that this error is only $O(\|\delta x\|^2)$, and not, as desired, $o(\|\delta x\|^2)$. Two terms play a major role:

\begin{itemize}
 \item The first term $w_1(\delta x):=\Li'(0_x,p_x)\Phi_x''(0_x)(\delta x,\delta x)$ occurs if $(\Rd,\Rp)$ is not second order consistent. We observe, however, that this term vanishes at a point that satisfies the first order optimality conditions and is small in a neighbourhood thereof due to~\eqref{eq:LLipschitz}. Thus, although this term is quadratic, we have, close to a local minimizer $x_*$ (using again $\mathbf{x}_*:= \Rp^{-1}(x_*))$: 
 \[
   w_1(\delta x)\sim \|\Li'(0_x,p_x)\|_{x,*}O(\|\delta x\|_x^2)\sim O(\|0_x-\mathbf{x}_*\|_x)O(\|\delta x\|_x^2).
 \]
\item The second term $w_2(\delta n) := p_x\Psi_y''(0_y)(\Ci'(0_x)\delta n,\Ci'(0_x)\delta n)$ only affects normal directions, but it does not vanish at a point that satisfies the first order optimality conditions. So it may affect the acceptance criteria of a globalization scheme and slow down local convergence, because certainly 
\[
w_2(\delta n) \sim O(\|\delta n\|_x^2).
\]
In practice, one often observes $\|\delta n\|_x\ll \|\delta x\|_x$ during local convergence, which would make this term small again, relative to $\|\delta x\|^2_x$, but this behaviour is hard to guarantee theoretically. 
\end{itemize}

For later reference, we state the following particular case:
\begin{prop}\label{cubicBoundModel1}
 Assume that \eqref{eq:Lipschitzf}, \eqref{eq:Lipschitzc1}, $\eqref{eq:LipschitzL}$, and \eqref{eq:boundedRetraction} hold at $x\in X$ with constants $\omega_{\Fp'}$, $\omega_{\Cp}$, $\omega_{\Lp}$ and $M_\Phi$. Assume that $(\Sd,\Sp)$ is second order consistent.
Then for sufficiently small $\delta x$ and simplified normal step $\delta s$ as defined in $(\ref{simplifiednormstepcharts})$ there is a constant $\omega_{\Fp}$ such that we have the estimate:
\begin{align*}
\vert \Fp(\delta x+\delta s)- \qd(\delta x) \vert
&\le \frac{\omega_{\Fp}}{6}\|\delta x\|_x^3+\|\Li'(0_x,p_x)\|_{x,*} M_\Phi\|\delta x\|_x^2. 
\end{align*}
\end{prop}

\subsection{A hybrid second order model}

We have seen that error of the standard quadratic model $\qd(\delta x)$ with respect to $\Fp(\delta x+\delta s)$ is only $O(\|\delta x\|_x^2)$, if the stratifications $(\Sd[y],\Sp[y])$ are not second order consistent. This is unsatisfactory, because second order consistency of $(\Sd[y],\Sp[y])$ is not required for fast local convergence of a plain SQP method. 
We will overcome this difficulty by introducing a hybric model $\qt$ that exhibits the desired properties. 

If $(\Sd[y],\Sp[y])$ is second order consistent, then we can use $\qd$ as a model. 
However, for the case when $(\Sd[y],\Sp[y])$ are not second order consistent, we propose to use the following hybrid model for $\delta n=\nu \Delta n$ with $\nu\in ]0,1]$ (a damping factor, which will be used for globalization):
\begin{equation}\label{eq:surrogateModel}
\begin{split}
 \qt(\delta n)(\delta t) &:= \Lp(\delta n,p_x)-(1-\nu)p_x\C(0_x)+(\Fi'(0_x)+\Ld''(0_x,p_x)\delta n)\delta t+\frac12 \Ld''(0_x,p_x)(\delta t,\delta t)\\
 &=\Fp(\delta n)+p_x(\Cp(\delta n)-(1-\nu)\C(0_x))+(\Fi'(0_x)+\Ld''(0_x,p_x)\delta n)\delta t+\frac12 \Ld''(0_x,p_x)(\delta t,\delta t).
 \end{split}
\end{equation}
We will see that this yields a better approximation of $\Fp$ along normal direction. 
\begin{lemma}\label{surrogate-model}
  For the surrogate model $\qt$, we have that:
 \begin{align}\label{eq:surrogate1}
  \qt(\delta n)(\delta t)-\qd(\delta x)
  &=\rp(\delta n)+\frac12\left(\Li'(0_x,p_x)\Phi_x''(0_x)(\delta n,\delta n)+p_x\Psi_y''(0_y)(\Ci'(0_x)\delta n,\Ci'(0_x)\delta n)\right). 
  \end{align}
  In particular, for fixed $\delta n$:  
  \[
   \underset{\delta t \in \ker\,\Ci'(0_x)}{\mathrm{argmin}}\,\qt(\delta n)(\delta t)=\underset{\delta t \in \ker\,\Ci'(0_x)}{\mathrm{argmin}}\,\qd(\delta n+\delta t).
  \]
  \end{lemma}
\begin{proof}
By definition of $\qd(v)$ we obtain,  using the fact that $\nu p_x\C(0_x)=-p_x \Ci'(0_x)\delta n = \Fi'(0_x)\delta n $
 \begin{align*}
 \Lp(\delta n,p_x)&-\LL(0_x,p_x) + \qd(\delta x) - \frac12 \Ld''(0_x,p_x)\delta n^2 \\
 &=\Lp(\delta n,p_x)- \Fi(0_x)- p_x\C(0_x) +\Fi(0_x)+\Fi'(0_x)\delta x + \frac12 \Ld''(0_x,p_x)\delta x^2 - \frac12 \Ld''(0_x,p_x)\delta n^2\\
   &= \Lp(\delta n,p_x)+(\nu-1)p_x\C(0_x)+ \Fi'(0_x)\delta t + \frac12 \Ld''(0_x,p_x)( \delta x+ \delta n, \delta t ) =\qt(\delta n)(\delta t).
 \end{align*}
Taking into account
\begin{align*}
 \Lp(\delta n,p_x)-\LL(0_x,p_x)&=\rp(\delta n)+\Li'(0_x,p_x)\delta n+\frac12 \Lp''(0_x,p_x)\delta n^2
 =\rp(\delta n)+\frac12 \Lp''(0_x,p_x)\delta n^2
\end{align*}
we obtain:
\begin{align*}
 \qt(\delta n)(\delta t)-\qd(\delta x)=\rp(\delta n)+\frac12\left(\Lp''(0_x,p_x)\delta n^2-\Ld''(0_x,p_x)\delta n^2\right)
\end{align*}
Now~\eqref{eq:differenceOfLxx} yields \eqref{eq:surrogate1}.
Obviously, the difference in \eqref{eq:surrogate1} is independent of $\delta t$. So $\qd$ and $\qt$ have the same minimizer in tangential direction.  
\end{proof}

\begin{lemma}\label{surrogate-estimate}
 For the surrogate model $\qt$, we have the identity
 \begin{align*}
  \Fp(\delta x +\delta s)- \qt(\delta n)(\delta t) =&\rp(\delta x) - \rp(\delta n) + \ssp(\delta x)+ \frac12 \Li'(0_x,p_x)\Phi_x''(0_x)(\delta t,\delta x+\delta n).
 \end{align*}
\end{lemma}
\begin{proof}
By Lemma~\ref{ModelEstimate} and Lemma~\ref{surrogate-model} we compute
\begin{align*}
 \Fp(\delta x+\delta s)-\qt(\delta n)(\delta t)&=(\Fp(\delta x+\delta s)- \qd(\delta x))-(\qt(\delta n)(\delta t)-\qd(\delta x))\\
 =\rp(\delta x)+\ssp(\delta x)&+\frac12\left(\Li'(0_x,p_x)\Phi_x''(0_x)(\delta x,\delta x)+p_x\Psi_y''(0_y)(\Ci'(0_x)\delta n,\Ci'(0_x)\delta n)\right)\\
 -\rp(\delta n)&-\frac12\left(\Li'(0_x,p_x)\Phi_x''(0_x)(\delta n,\delta n)+p_x\Psi_y''(0_y)(\Ci'(0_x)\delta n,\Ci'(0_x)\delta n)\right)\\
 &=\rp(\delta x)+\ssp(\delta x)-\rp(\delta n)+\frac12 \Li'(0_x,p_x)\left(\Phi_x''(0_x)(\delta x,\delta x)-\Phi_x''(0_x)(\delta n,\delta n)\right). 
\end{align*}
The crucial observation is that $p_x\Psi_y''(0_y)(\Ci'(0_x)\delta n,\Ci'(0_x)\delta n)$ cancels out. Finally, since $\Phi_x''$ is bilinear, we compute
\[
 \Phi_x''(0_x)(\delta x,\delta x)-\Phi_x''(0_x)(\delta n,\delta n)=\Phi_x''(\delta x-\delta n,\delta x+\delta n),
\]
which yields the stated result. 
%
%
\end{proof}

Combination with Lemma~\ref{prop:LipschitzConsistency} yields the following result:
\begin{prop}\label{cubicBoundModel}
 Assume that \eqref{eq:Lipschitzf}, \eqref{eq:Lipschitzc1}, $\eqref{eq:LipschitzL}$, and \eqref{eq:boundedRetraction} hold at $x\in X$ with constants $\omega_{\Fp'}$, $\omega_{\Cp}$, $\omega_{\Lp}$ and $M_\Phi$.
Then for sufficiently small $\delta x$ and simplified normal step $\delta s$ as defined in $(\ref{simplifiednormstepcharts})$ we have the estimate:
\begin{align*}
\vert \Fp(\delta x+\delta s)- \qt(\delta n )(\delta t) \vert
&\leq  \frac{\omega_{\Fp}}{6}\|\delta x\|_x^3+\frac{1}{2}\vert\Li'(0_x,p_x)\Phi_x''(0_x)(\delta t,\delta x+\delta n) \vert\\
&\le \frac{\omega_{\Fp}}{6}\|\delta x\|_x^3+\|\Li'(0_x,p_x)\|_{x,*} M_\Phi\|\delta x\|_x^2. 
\end{align*}
\end{prop}
\begin{proof}

This follows directly from Lemma~\ref{prop:LipschitzConsistency} and Lemma~\ref{surrogate-estimate} by adding up the relevant terms. 
\end{proof}

\paragraph{Discussion.}
We conclude that $\qt$ has all the properties we need:
\begin{itemize}
 \item By Lemma~\ref{surrogate-model} minimizers of $\qt$ for given $\delta n$ are the tangent steps $\Delta t$, defined in \eqref{Tangential_STP} and can thus be computed as before.
  \item By Proposition~\ref{cubicBoundModel} the consistency properties of $\qt$ are superior to the ones of $\qd$. The reason is that $\Li'(0_x,p_x)$ becomes small close to a point that satisfies the first order optimality conditions, rendering the error term quasi second order locally, just as discussed in the last subsection. We will show below that this is sufficient for transition to fast local convergence. 

 \item  As a slight disadvantage, the use of $\qt$ implies an additional computational cost, namely the evaluation of $\Lp(\delta n,p_x)$. In practice, this is negligible, compared to the remaining computational efforts. Also, evaluation of $\Fp$ and $\Cp$ at points other than $0_x$ has to be implemented anyway, so the evaluation of $\Lp(\delta n,p_x)$ can be added easily to existing code. 

  \end{itemize}

%
%

\section{Globalization by a Composite Step Method}\label{sec:global}

In the last section we established results, which are important ingredients for showing transition to fast local convergence of any SQP-globalization scheme. In this section, we will demonstrate for a specific example of a globalization scheme that transition to fast local convergence can indeed be shown. 

Classical globalization strategies in equality constrained optimization on linear spaces rely on a combination of functional descent and reduction of residual norm $\|c(x)\|$.  For example, merit functions of the form $\phi(x) := f(x)+\kappa\|c(x)\|$ weigh both quantities with a parameter $\kappa$ that is adjusted adaptively. In our setting, where $c(x)\in Y$ is an element of a manifold, $\|c(x)\|$ cannot be used straightforwardly. One possible surrogate would be to compute the geodesic distance $d(c(x),y_*)$, which however, would involve computations of geodesics on $Y$. Another, more practical way would be to use stratifications and norms $\|0_{c(x)}-\ystar\|_{c(x)}$ on $T_{c(x)} Y$. This means, however, that the norm changes in each step of the algorithm, which may pose additional difficulties in a proof of global convergence.   

In~\cite[Section 4]{lubkoll2017affine} a globalization scheme has been proposed for an affine covariant composite step method. It is based on the idea of affine covariant Newton methods, a globalization strategy that relies solely on evaluation of norms in the space of iterates (cf. \cite{deuflhard2011newton}). While lacking a rigorous proof of global convergence, this strategy has a clear theoretical motivation and has shown very robust behaviour in practical problems. 

\subsection{Affine covariant globalization}

In the following we will recapitulate the main features of the affine covariant composite step method, introduced in \cite{lubkoll2017affine}, and adjust it to the case of manifolds, where necessary. Since our aim is to study local convergence of our algorithm, we concentrate on the aspects of our scheme that are relevant for local convergence. Readers, interested in further algorithmic details are referred to \cite{lubkoll2017affine}.

Each step of the globalization scheme at a current  iterate $x$ will be performed on $T_x X$ and $T_y Y$, where $y=c(x)$. In accordance with Assumption~\ref{ass:retractions} we use local retractions and stratifications $\Rd,\Sd$ to compute local models,
a local retractions $\Rp : V_x \to X$ and a global stratification $\Sp$ to compute residuals and updates. We assume that $V_x$ is closed. In practice, $V_x=T_x X$ is desirable, and for simplicity of compuations, $V_x$ should be convex. 
Then the globalization scheme from~\cite{lubkoll2017affine} can be applied to the pullback of the problem. 

As elaborated in \cite{lubkoll2017affine} we use the algorithmic parameters 
$[\omega_{\Cp}]$ to capture the nonlinearity of $c$, and $[\omega_{\Fp}]$ to capture the nonlinearity of $f$. The square brackets indicate that these a-posteriori are estimates for corresponding theoretical quantities $\omega_{\Cp}$ and $\omega_{\Fp}$, which occur in~\eqref{eq:Lipschitzc1}, and Propositions~\ref{cubicBoundModel1} and~\ref{cubicBoundModel}. Simply speaking, large values of these parameters indicate a highly nonlinear problem, so a globalization scheme should restrict step sizes to avoid divergence.  Initial estimates have to be provided for these parameters, which will be updated in each step (similar to a trust-region parameter).

\paragraph{Computation of a composite step.}
We start by computing the full normal step $\Delta n$, and compute a maximal damping factor $\nu \in ]0,1]$ and $\delta n := \nu \Delta n$, such that
\begin{equation}\label{eq:ellbow}
 \frac{[\omega_{\Cp}]}{2}\|\delta n\|_x=\frac{[\omega_{\Cp}]}{2}\nu\|\Delta n\|_x \le \rho_{\rm ellbow}\Theta_{\rm aim}, \; \mbox{ and }\; \delta n\in V_x 
\end{equation}
Here $\Theta_{\rm aim}\in ]0,1[$ is a desired Newton contraction for the underdetermined problem $\Cp(x)-\ystar=0$
and $\rho_{\rm ellbow}\in ]0,1]$ provides some ellbow space in view of the last line of \eqref{eq:pullbackedProblem}, below.
Then, $p_x$ is computed via~\eqref{eq:pKKT}, so that $\Delta t$ can be computed via~\eqref{Tangential_STP}. If $\Ld''(0_x,p_x)$ does not satisfy~\eqref{eq:ellipticSSC}, then a suitably modified solution (e.g. via truncated cg or  Hessian modification) can be used.  

For a given current iterate $x\in X$, for algorithmic parameters $[\omega_{\Fp}]$ and $[\omega_{\Cp}]$, and after $\delta n := \nu \Delta n$, $p_x$, and $\Delta t$ have been computed, we solve the following problem in $\tau \in \R$ to obtain a trial correction $\delta x=\nu \Delta n
+\tau \Delta t$: 
\begin{align}
\begin{split}\label{eq:pullbackedProblem}
\min_{\tau : \delta x=\nu \Delta n+\tau \Delta t}  \Fi(0_x)+\Fi'(0_x)\delta x +\frac{1}{2} \Ld''(0_x,p_x) (\delta x,\delta x) &+\frac{[\omega_{\Fp}]}{6} \| \delta x \|_x^3 \\
\, s.t. \,\, \nu\Cp(0_x)+\Ci'(0_x)\delta x=&0, \\
\frac{[\omega_{\Cp}]}{2}\|\delta x \|_x\leq \Theta_{\rm aim} \mbox{ and } \delta x\in V_x
\end{split}
\end{align}
 This one-dimensional problem is easy to solve. More sophisticated strategies to compute $\delta t$ directly as an approximate minimizer of the cubic model are conceivable and have been described in the literature (cf. e.g.~\cite{Cartis2011}). We observe that the step is restricted in two ways to achieve globalization: the third line of \eqref{eq:pullbackedProblem} restricts $\|\delta x\|_x$ by a trust-region like constraint, taking into account the nonlinearity of $\Cp$, represented by $[\omega_{\Cp}]$, which promotes feasibility. The cubic term in the first line penalizes large steps, taking into account the nonlinearity of $\Fp$, represented by $[\omega_{\Fp}]$ and thus promotes descent in the objective.  

\paragraph{Update of algorithmic quantities.} After $\delta x = \delta n+\tau \Delta t$
has been computed as a minimizer of~\eqref{eq:pullbackedProblem}, the simplified normal step 
$\delta s$ is computed via~\eqref{simplifiednormstepcharts}. Moreover, if $V_x \neq T_x X$, we compute the largest $\sigma \in [0,1]$, such that $\delta x+\sigma \delta s \in V_x$. 
At this point $[\omega_{\Cp}]$ and $[\omega_{\Fp}]$ can be updated. In view of \eqref{eq:estds} we define just as in~\cite{lubkoll2017affine}:
\begin{equation}\label{eq:estimateomegac}
 [\omega_{\Cp}]^{\rm new} := \frac{2}{\|\delta x\|_x^2}\|\delta s\|_x
\end{equation}
as an algorithmic quantity that locally estimates the affine covariant Lipschitz constant $\omega_{\Cp}$ defined in \eqref{eq:Lipschitzc1}. This update rule and  \eqref{eq:estds} imply that  $[\omega_{\Cp}]\le \omega_{\Cp}$ during the algorithm.  

Concerning $[\omega_{\Fp}]$, the use of retractions requires a modification, compared to~\cite{lubkoll2017affine}. We first define
\[
 \q(\delta x) := \left\{
 \begin{array}{rcl}
    \qd(\delta x) & :& \mbox{if }(\Sd,\Sp) \mbox{ is second order consistent }\\
    \qt(\delta n)(\delta t) & :&   \mbox{otherwise }
 \end{array}
 \right.
\]
and then set, in view of Proposition~\ref{cubicBoundModel1} and~\ref{cubicBoundModel}:
\[
 [\omega_{\Fp}]^{\rm raw} := \frac{6}{\|\delta x\|_x^3}(\Fp(\delta x+\sigma \delta s)-\q(\delta x)) 
\]
This (potentially negative) estimate has to be augmented by some save-guard bounds of the form
\[
 [\omega_{\Fp}]^{\rm new} = \min\{\overline b [\omega_{\Fp}]^{\rm old},\max\{\underline b [\omega_{\Fp}]^{\rm old}, [\omega_{\Fp}]^{\rm raw}\}\}
\]
with $0 < \underline b < 1 < \overline b$. This prohibits that $[\omega_{\Fp}]$ is changed too rapidly. 
 
\paragraph{Acceptance criteria.} For acceptance of iterates, we perform a contraction test and a decrease test. The contraction test requires, just as in~\cite{lubkoll2017affine},
\begin{equation}\label{eq:acceptanceC2}
 \frac{\|\delta s\|_x}{\|\delta x\|_x} \le \Theta_{\rm acc} 
\end{equation}
for acceptance, with some parameter $\Theta_{\rm acc} \in ]\Theta_{\rm des},1[$. For a short motivation, consider the case $\nu =1$. Then $\delta s = -\Ci'(0_x)^{-}\Cp(\delta x)$, and $\delta s$ can be seen as the second step of a \emph{simplified Newton method}, applied to the problem $\Cp(\delta x)=0$. 
Hence, the left hand side of \eqref{eq:acceptanceC2} can be seen as an estimate the ratio of contraction of this method and thus gives an indication if $x$ is close to being feasible. The choice of $\Theta_{\rm acc}$ imposes a requirement on this local rate of contraction. If $\Theta_{\rm acc}\approx 0$ the algorithm tends to stay in a small neighbourhood of the feasible set, taking smaller steps, while for $\Theta_{\rm acc}\approx 1$ the algorithm operates in a larger neighbourhood of the feasible set and takes more aggressive steps. 
Again, we refer to the detailed exposition~\cite{deuflhard2011newton} on affine covariant Newton methods and to~\cite{lubkoll2017affine} for a thorough discussion.

For the decrease test
we define the cubic model
\[
 \mathbf{m}_{[\omega_{\Fp}]}(v) := \q(v)+\frac{[\omega_{\Fp}]}{6}\|v\|_x^3.
\]
and require a ratio of actual decrease and predicted decrease condition. We choose $\underline \eta \in ]0,1[$ and define:
\begin{equation}\label{eq:decrease}
 \eta := \frac{\Fp(\delta x+\sigma \delta s)-\mathbf{m}_{[\omega_{\Fp}]}(\delta n)}{\mathbf{m}_{[\omega_{\Fp}]}(\delta x)-\mathbf{m}_{[\omega_{\Fp}]}(\delta n)}, 
\end{equation}
observe by \eqref{eq:pullbackedProblem} that the denominator is always negative, unless $\tau = 0$, which only occurs, if $\mathbf{f'}(0_x)=0$ on $\ker\,\Ci'(0_x)$. 
Then we require
\begin{equation}\label{eq:acceptanceD}
 \eta \ge \underline \eta \quad \mbox{ for some } \underline \eta \in ]0,1[
\end{equation}
for acceptance of the step.
As a slight modification to \cite{lubkoll2017affine} we increase $[\omega_{\Fp}]$ at least by a fixed factor $\hat b \in ]1,\overline b]$ with respect to $[\omega_{\Fp}]^{\rm old}$, if the decrease condition~\eqref{eq:acceptanceD} fails. Moreover, $[\omega_{\Fp}]$ will not be increased, if $\eta \ge \hat \eta$ for some $\hat \eta \in [\underline \eta,1[$ which is usually chosen close to $1$.

\begin{algorithm}[h!]
  \caption{Affine covariant composite step method (inner loop slightly simplified)}\label{alg:inner_loop_rough}
  \begin{algorithmic}
  \Require initial iterate $x$, $[\omega_{\Cp}], [\omega_{\Fp}]$
  \Repeat // \emph{NLP loop}
  \State compute linear-quadratic models of $f$ and $c$ based on $(\Rd,\Sd)$
  \Repeat // \emph{step computation loop}
    \State compute normal step $\Delta n$ via \eqref{eq:minimumnormKKT}, using  $\Sp$
    \State compute Lagrange multiplier estimate $p_x$ via \eqref{eq:pKKT}
    \State compute maximal $\nu\in ]0,1]$, such that \eqref{eq:ellbow} holds
    \State compute tangent step $\Delta t$ via~\eqref{Tangential_STP}
    \State compute $\tau$ via~\eqref{eq:pullbackedProblem} and composite step $\delta x=\nu \Delta n+\tau \Delta t$
    \State compute second order correction $\delta s$, via~\eqref{simplifiednormstepcharts}, using $(\Rp,\Sp)$
    \State evaluate $\Fp(\delta x+\sigma \delta s)$
    \State compute maximal $\sigma \in [0,1]$, such that $\delta x+\sigma \delta s\in V_x$ 
    \State evaluate acceptance tests \eqref{eq:acceptanceC2} and \eqref{eq:acceptanceD}
    \State update Lipschitz estimates $[\omega_{\Cp}], [\omega_{\Fp}]$, using $\delta s$, $\Fp(\delta x+\sigma \delta s )$, and $\q(\delta x)$
    \Until{ $\delta x$ accepted }
    
   $x \gets \Rp(\delta x+\sigma \delta s)$
  \Until{ converged }
  \end{algorithmic}
 \end{algorithm}

 We obtain Algorithm~\ref{alg:inner_loop_rough}, a slightly simplified version of the method, proposed in~\cite{lubkoll2017affine}. Some additional precautions are necessary to guarantee that the inner loop is terminated after finitely many steps. These details, however are not relevant for the following local convergence study. 
 
\subsection{Transition to fast local convergence}

Close to a local minimizer we study, if the computed steps $\delta x_k$ approach to the full Lagrange-Newton steps $\Delta x_k$ asymptotically, and if they inherit local quadratic convergence from these.  
\begin{theo}
 Suppose that Assumption~\ref{ass:local1} and Assumption~\ref{ass:local2} hold in a neighbourhood of a local minimizer $x_{*}$ and assume that Algorithm~\ref{alg:inner_loop_rough} yields a sequence $x_k$ that converges to $x_*$. Then convergence is quadratic in the sense of Theorem~\ref{thm:localNewton}. 
\end{theo}
\begin{proof}
 As $x_k \to x_{*}$, we obtain $\|\delta n_k\|_{x_k} \le \|\delta x_k\|_{x_k}\to 0$, because our assumptions include non-degeneracy~\eqref{eq:primalNormEquivalence}.
  We will show that the damping factors $\nu_k$ and $\tau_k$ tend to 1 as $k\to \infty$. 
 
 From Lemma~\ref{prop:LipschitzConsistency} and \eqref{eq:estimateomegac} we infer that $[\omega_{\Cp}]\le\omega_{\Cp}$ remains bounded.
 Hence, it follows from~\eqref{eq:ellbow} that the normal damping factor $\nu_k$ becomes $\nu_{k}=1$  after finitely many steps, and thus $\delta n_k=\Delta n_k$ for all $k$ greater than some $K_0$. Due to Lemma~\ref{prop:LipschitzConsistency} the acceptance test \eqref{eq:acceptanceC2} is passed, if $\delta x$ is sufficiently small. This happens after finitely many steps, because $\delta x_k \to 0$. Thus, in the following we may assume that $k$ is sufficiently large to have $\nu_k=1$ and $\delta x_k$ satisfies the third line of \eqref{eq:pullbackedProblem} and passes  \eqref{eq:acceptanceC2}.

 It remains to show $\lim_{k\to \infty}\tau_k=1$. 
 Since $\Delta x_k$ minimizes $\qd$ on $\Delta n_k+\ker \Ci'(0_{x_k})$, we conclude that $\qd'(\Delta x_k)$ vanishes on $\ker \Ci'(0_{x_k})$. This implies:
\begin{align*}
0=\qd'(\Delta x_k)\Delta t_k&=(\Fi'(0_{x_k})\Delta t_k+\Ld''(0_{x_k},{p}_{x_k})(\Delta x_k,\Delta t_k)\\
&=(\Fi'(0_{x_k})+\Ld''(0_{x_k},{p}_{x_k})\Delta n_k)\Delta t_k+\Ld''(0_{x_k},{p}_{x_k})(\Delta t_k,\Delta t_k). 
\end{align*}
In addition, $\tau_k$ is computed as the minimizer of the first line of \eqref{eq:pullbackedProblem} along the direction $\Delta t_k$. Hence the derivative of this term in direction $\Delta t_k$ vanishes at $\delta x_k=\Delta n_k+\tau \Delta t_k$ and we compute, exploiting $\langle \Delta n_k,\Delta t_k\rangle_{x_k}=0$:
\begin{align*}
0&= \Fi'(0_{x_k})\Delta t_k +\Ld''(0_{x_k},{p}_{x_k})(\delta x_k,\Delta t_k) +\frac{[\omega_{\Fp}]}{2}\|\delta x_{k} \|_{x_k}\left\langle \delta x_{k},\Delta t_k \right\rangle_{x_k} \\ 
&=(\Fi'(0_{x_k})+ \Ld''(0_{x_k},{p}_{x_k})\Delta n_k)\Delta t_k 
 +\tau_k\Big(\Ld''(0_{x_k},{p}_{x_k})(\Delta t_k,\Delta t_k)  +\frac{[\omega_{\Fp}]}{2}\|\delta x_{k} \|_{x_k}\|\Delta t_k\|^2_x\Big).
\end{align*} 
Subtracting these two equations, we obtain
\begin{align}\label{derivative-CubicModel}
\Ld''(0_{x_k},{p}_{x_k})(\Delta t_k,\Delta t_k)=\tau_k\left(\Ld''(0_{x_k},{p}_{x_k})(\Delta t_k,\Delta t_k)+ \frac{[\omega_{\Fp}]}{2}\|\delta x_{k} \|_{x_k}\|\Delta t_k\|_{x_k}^2\right)
\end{align}
  and thus a formula the the damping factor:
  \begin{align}\label{eq:esttau}
  \tau_k=\frac{\Ld''(0_{x_k},{p}_{x_k})(\Delta t_k,\Delta t_k)}{\Ld''(0_{x_k},{p}_{x_k})(\Delta t_k,\Delta t_k)+ \frac{[\omega_{\Fp}]}{2}\|\delta  x_{k} \|_{x_k}\|\Delta t_k\|_{x_k}^2}\ge \frac{1}{1+\frac{[\omega_{\Fp}]}{2\alpha_{\Ld''}}\|\delta  x_{k} \|_{x_k}} 
  \end{align}
  where $\alpha_{\Ld''}$ is the ellipticity constant of $\Ld''(0_x,p_x)$ due to \eqref{eq:ellipticity}. 
  Sufficiently close to $x_*$ we infer by Proposition~\ref{pro:localNewton} and the triangle inequality: 
  \begin{equation*}
    \|0_{x_k}-\mathbf{x}_*\|_{x_k} \le \|\Delta x_k\|_{x_k}+\|0_{x_k}+\Delta x_k-\mathbf x_*\|_{x_k}\le \|\Delta x_k\|_{x_k}+\beta \|0_{x_k}-\mathbf{x}_*\|_{x_k}^2,
  \end{equation*}
  which implies, together with \eqref{eq:esttau}, and $\tilde C \ge (1-\beta \|0_x-\mathbf x_*\|_{x_k})^{-1}$:
  \begin{equation}\label{eq:dxvserror}
  \begin{split}
    \|0_{x_k}-\mathbf{x}_*\|_{x_k} &\le \tilde C\|\Delta x_k\|_{x_k} \le \frac{\tilde C}{\tau_k}\|\delta x_k\|_{x_k}
    \le \tilde C\left(1+\frac{[\omega_{\Fp}]}{2\alpha_{\Ld''}}\|\delta x_k\|_{x_k}\right)\|\delta x_k\|_{x_k}.
    \end{split}
  \end{equation}
  
  Next, consider the acceptance test~\eqref{eq:decrease}. Since 
  $\mathbf{m}_{[\omega_{\Fp}]}(\delta x_k) < \mathbf{m}_{[\omega_{\Fp}]}(\Delta n_k)$, \eqref{eq:decrease} is
  certainly fulfilled with $\eta \ge 1$, if $\F(\delta x_k+\delta s_k) \le  \mathbf{m}_{[\omega_{\Fp}]}(\delta x_k)$. To establish such an estimate, we compute from Proposition~\ref{cubicBoundModel1}, Proposition~\ref{cubicBoundModel}, and \eqref{eq:dxvserror}:
  \begin{equation}\label{eq:nonmaratos}
  \begin{split}
   \Fp(\delta x_k+\delta s_k)- \q(\delta x_k) &\le \frac{\omega_{\Fp}}{6}\|\delta x_k\|_{x_k}^3+\omega_*\|0_{x_k}-\mathbf{x}_*\|_{x_k}M_\Phi\|\delta x_k\|_{x_k}^2\\
    &\le C\left(1+\frac{[\omega_{\Fp}]}{2\alpha_{\Ld''}}\|\delta x_k\|_{x_k}\right)\|\delta x_k\|_{x_k}^3.
    \end{split}
  \end{equation}
Since 
\begin{equation*}
\mathbf{m}_{[\omega_{\Fp}]}(\delta x)-\q(\delta x_k)=\frac{[\omega_{\Fp}]}{6}\|\delta x_k\|_{x_k}^3
\end{equation*}
we obtain $\Fp(\delta x_k+\delta s_k) \le \mathbf{m}_{[\omega_{\Fp}]}(\delta x_k)$, if
\[
  C\left(1+\frac{[\omega_{\Fp}]}{2\alpha_{\Ld''}}\|\delta x_k\|_{x_k}\right) \le \frac{[\omega_{\Fp}]}{6}
\]
For sufficiently small $\delta x_k$ this is true, if
\begin{equation}\label{eq:boundomega}
 [\omega_{\Fp}] \ge \frac{6C}{1-\frac{3C}{\alpha_{\Ld''}}\|\delta x_k\|_{x_k}}. 
\end{equation}
Thus, we conclude that close to a minimizer \eqref{eq:decrease} always holds with $\eta \ge 1 >\hat \eta$, if $[\omega_{\Fp}]$ is above the bound, given in \eqref{eq:boundomega}, which only depends on the problem and the chosen neighbourhood around $x_*$. Consequently, by our algorithmic mechanism, 
$[\omega_{\Fp}]$ cannot become unbounded. 
Hence, as $x_k \to x_{*}$, implies by \eqref{eq:esttau} that $\tau_k \to 1$ because $[\omega_{\Fp}]\|\delta x\|_{x_k}/\alpha_{\Ld''}\to 0$. Thus, we obtain local superlinear convergence of our algorithm. More accurately, by boundedness of $[\omega_{\Fp}]$ we obtain, using $\|\delta x_k\|_{x_k}\le \|\Delta x_k\|_{x_k}$ and \eqref{eq:esttau}:
\[
 \tau_k \ge \frac{1}{1+c\|\Delta x_k\|_{x_k}}\quad \Rightarrow\quad  1-\tau_k \le \tilde c\|\Delta x_k\|_{x_k}
\]
and hence
\[
 \|\Delta x_k-\delta x_k\|_{x_k}\le (1-\tau_k)\|\Delta x_k\|_{x_k}\le \tilde c\|\Delta  x_k\|_{x_k}^2.
\]
Since also $\|\delta s_k\|_{x_k}\le \omega_{\Cp}\|\Delta x_k\|_{x_k}^2/2$, we have
\[
 \|\Delta x_k-(\delta x_k+\delta s_k)\|_{x_k}\le \hat c\|\Delta x_k\|_{x_k}^2,
\]
so the error between full Newton step and modified step decreases quadratically. Thus, quadratic convergence of the full Newton method in Proposition~\ref{pro:localNewton} and Theorem~\ref{thm:localNewton} carries over to our globalized version. 
\end{proof}

\begin{remark}
Observe, how the result $|\Fp(\delta x+\delta s)-\q(\delta x)|=O(\|\delta x\|_x^3)$ enters the proof. If, instead of \eqref{eq:nonmaratos} only $|\Fp(\delta x+\delta s)- \q(\delta x)|\le C\|\delta x\|_{x_k}^2$ holds, then \eqref{eq:boundomega} has to be replaced by $[\omega_{\Fp}]\ge C/\|\delta x_k\|_{x_k}$, so we only can only expect boundedness of $[\omega_{\Fp}]\|\delta x_k\|_{x_k}$.
Then, however, \eqref{eq:esttau} does not imply $\tau_k \to 1$ anymore, and thus fast local convergence cannot be shown. In a similar way other globalization schemes are affected, as well. This predicted slow-down of local convergence has also been observed in computations, historically. This is called the Maratos effect, and led to the indicated algorithmic developments.
\end{remark}

\section{Application: an inextensible flexible rod}\label{flexible_rod}

In this section we consider the numerical simulation of flexible inextensible rods to illustrate our approach and to demonstrate its viability. In particular, we highlight some of the theoretically observed robustness properties with respect to the use of two different retractions. 

Flexible rods are present in many real life problems, for 
example engineers are interested in the static and dynamic behaviour of flexible pipelines used in off-shore oil production under the 
 effects of streams, waves, and obstacles; or protein structure comparison \cite{liu2011mathematical} where elastic 
elastic curves are used to represent and compare protein structures. Here we consider the  problem where the stable equilibrium position of an inextensible transversely
isotropic elastic rod under dead load is sought. First we provide the formulation and the mathematical analysis of the problem, followed by the discretization and the derivatives of the 
mappings over the manifold of kinematically admissible configurations.

\subsection{Problem formulation}

Here we provide the energetic formulation of the problem of finding the stable equilibrium position of an inextensible, transversely isotropic
elastic rod under dead loading. For more details on the derivation of the model see \cite{glowinski1989augmented}. We start with the energy minimization problem
\begin{align}\label{min_prob}
\min_{y\in M} J(y)
\end{align}
where the energy $J$ and the manifold $M$ which describes the inextensibility condition are given by:
\begin{align*}
J(y)&=\frac{1}{2}\int_{0}^{1} \sigma(s) \left\langle y''(s),y''(s)\right\rangle ds -\int_0^1 \langle g(s),y(s) \rangle\, ds, \\
M&=\{y\mid y\in H^2([0,1];\mathbb{R}^3),\vert y'(s)\vert =1  \, \mbox{on} \, [0,1] \}. 
\end{align*}
Boundary conditions are given by 
\begin{align}\label{B.C}
\begin{split}
y(0)&=y_a \in \mathbb{R}^3, \,\,\, y'(0)=y'_a\in \mathbb{S}^2 \\
y(1)&=y_b \in \mathbb{R}^3, \,\,\, y'(1)=y'_b\in \mathbb{S}^2. 
\end{split}
\end{align}  
The quantity $\overline \sigma > \sigma(s)\ge \underline \sigma>0$ is the flexural stiffness of the rod, $g$ is the lineic density of external loads, and $y'$, $y''$ are the derivatives of $y$ with respect to $s\in[0,1]$. 
Denote by $\mathbb{S}^2$ the 
unit sphere
\[
 \mathbb{S}^2=\{v\in \mathbb{R}^3:\,|v|=1 \}.
\]
Introducing $v(s):=y'(s)$ we reformulate \eqref{min_prob} as a mixed problem:
\begin{align*}
\min f(y,v)&:=\frac{1}{2}\int_{0}^{1} \sigma \left\langle v',v'\right\rangle ds -\int_0^1 \langle g,y\rangle \, ds\\
\mbox{ s.t. } c(y,v) &:= y'-v =0
\end{align*}
defined on
\begin{align}\label{boundary_conditions}
\begin{split}
Y&=\{y\in H^2([0,1];\mathbb{R}^3)\,\, : \,\, y(0)=y_a,\,\, y(1)=y_b  \} \\
V&=\{v\in H^1([0,1];\mathbb{S}^2)\,\, :\,\, v(0)=v_a,\,\, v(1)=v_b  \}.
\end{split}
\end{align}
In short we get a constrained minimization problem of the form:   
\begin{align}\label{eq:rodasminimization}
 \min_{(y,v)\in (Y\times V)} f(y,v) \,\,\, \mbox{ s.t }\,\,\, c(y,v)=0.
 \end{align}
In the following we discuss application of our algorithm to a discretized version of 
\eqref{eq:rodasminimization}. The advantage of this formulation is that $V$ can now be discretized as a product manifold. 

Concerning the study of existence and the uniqueness of the solutions of the problem $(\ref{min_prob})$ we refer the reader
to the books \cite{glowinski1989augmented,antman1978global} for a detailed and complete mathematical analysis of these kind of problems. In the following we assume that $\sigma\in L^{\infty}([0,1])$ is non-negative.
Concerning to the existence properties of the problem 
$(\ref{min_prob})$ we have the following theorem.
\begin{theo}
Suppose that $\vert y_a-y_b \vert < 1$, \eqref{B.C} holds, and that the linear functional $y \to \int_{0}^{1} \langle g, y\rangle ds$ is continuous on $H^2([0,1];\mathbb{R}^3)$. Then the problem \eqref{min_prob} has at least one solution.
\end{theo} 
\begin{proof}
See \cite{glowinski1989augmented}.
\end{proof}
%
%

\subsection{Finite difference approximation}
For discretization, we use a very simple finite difference approach. We discretize the interval $[0,1]$ uniformly 
$s_i=ih,\; i=0,\dots, n$
where $h=1/n$.
 Evaluating at each nodal point we denote $y(s_i)=y_i \in \mathbb{R}^3$ and $v(s_i)=v_i \in \mathbb{S}^2$ for $i=0\dots n$ with boundary conditions:
\begin{align*}
y(0)=y_0&=y_{a}\in \mathbb{R}^3, \qquad  y(1)=y_n=y_{b}\in \mathbb{R}^3, \\
v(0)=v_0&=v_{a} \in \mathbb{S}^2, \qquad v(1)=v_n=v_{b}\in \mathbb{S}^2.
\end{align*}
Employing forward finite difference discretization and a Riemann sum for the integrals yields the following approximation of the energy functional
\begin{align}\label{Functional}
f(y,v)=\frac{1}{2}\sum_{i=0}^{n-1} h\left\langle \frac{1}{h}(v_{i+1}-v_i) ,\frac{1}{h}(v_{i+1}-v_i) \right\rangle -\sum_{i=1}^{n} h \left\langle g_i,y_i \right\rangle.
\end{align}

Concerning the constraint $c(y,v)$, performing forward finite differences to the equation $y'-v=0$, the discretized constraint mapping takes the form
\begin{align}\label{constraint}
\frac{y_{i+1}-y_i}{h}-v_i=0, \quad i=0,...,n-1.
\end{align}
We observe that the codomain of our constraint mapping is a linear space, which eliminates the need for a retraction in the codomain. 

In the formulation above of the discrete inextensible rod and with $y_0,v_0,y_n,v_n$ fixed, the manifold we choose as our manifold $X$ the product manifold
\[
  X=(\mathbb{R}^3 \times \mathbb{S}^2)^{n-1},
\] 
where $n$ is the number of grid vertices. The elements of the manifold $X$ are 
denoted by the cartesian product
\begin{align*}
x=(y,v)=\prod_{i=1}^{n-1}(y_i,v_i), \hspace{0.5cm} y_i\in \mathbb{R}^3, \, v_i\in \mathbb{S}^2 \subset \R^3, 
\end{align*}
so each $v_i \in \mathbb S^2$ is represented by an element of $\R^3$ with unit norm. 

The tangent space at $x=(y,v)\in X$ is given by the following direct sum of vector spaces
\begin{align*}
T_{x}X=\bigoplus_{i=1}^{n-1}\left(T_{y_i}\mathbb{R}^3\oplus T_{v_i}\mathbb{S}^2 \right).
\end{align*}
The update, using the retraction map $R_{x}:T_{x}X \to X$, is done in a component-wise way by:
\begin{align*}
(y_{+},v_{+})=R_{x}(\delta y,\delta v)=\prod_{i=1}^{n-1}(y_i+\delta y_i, R^{\mathbb S^2}_{v_i}(\delta v_i) ).
\end{align*} 
Thus, due to the product structure we only have to provide and implement a retraction $R^{\mathbb S^2}_{v_i} : T_{v_i}\mathbb{S}^2\to \mathbb S^2$ to obtain a retraction on the product manifold $X$. 
\subsection{Retractions and their implementation via local parametrizations}\label{sec:localpara}

As presented above, retractions are defined as mappings $R_x : T_x X\to X$ between an abstract vector space $T_xX$ and a manifold $X$. For their implementation on a computer, we have to choose a basis of $T_xX$ and define a representation of $R_x$ with respect to that basis.  
This yields the concept of local parametrizations as described in \cite{absil2009optimization}. 

Let $\{\xi_1,\dots \xi_d\}$ be a basis of $T_x X$ so that each $v\in T_x X$ can be written uniquely as a linear combination of basis vectors:
\[
 v = \sum_{i=1}^d u_i\xi_i, \quad \mbox{ with } u \in \R^d.
\]
which yields a linear isomorphism $W_x:\R^d \to T_xX$, mapping $u$ to $v$. 

Then a local parametrization, based on a given retraction $R_x:T_xX\to X$ is defined as:
\begin{align*}
 \mu_x:\mathbb{R}^d&\to X\\
        u &\mapsto \mu_x(u) := R_x(W_x u)=R_x\left(\sum_{i=1}^d u_i\xi_i\right).
\end{align*}
This implies that $\mu_x(0)=x$ and that $\mu_x$ is a local diffeomorphism around $0\in \mathbb{R}^d$ with derivative $W_x$ at $0$. In fact, it is $\mu_x$ that can be implemented on a computer. The corresponding subroutine takes an element $u\in \R^d$ as an input and yields a suitable representation of $x=\mu_x(u) \in X$.  

For our case we construct local parametrizations around each node $v_i$ on each sphere, induced by retractions $R_{v_i}^{\mathbb{S}^2}$, this is, 
we look for local diffeomorphisms around $v_i\in \mathbb{S}^2$:
\begin{align*}
\mu_{v_i}:\mathbb{R}^2 &\to \mathbb{S}^2 \subset \R^3\\
        u &\to \mu_{v_i}(u)
 \end{align*}
 such that $\mu_{v_i}(0)=v_i$  and  $\mu_{v_i}(u)\in \mathbb{S}^2$. Elements of $\mathbb{S}^2$ are represented as vectors $v\in \R^3$ with $\|v\|_{\R^3}=1$ via the standard embedding $\mathbb{S}^2\subset \R^3$. Due to this embedding, we may view $\mu_{v_i}$ as a mapping $\R^2\to \R^3$ and compute ordinary first and second derivatives.  

\paragraph{Two alternative retractions.} 
 To study the effect of consistency of retractions numerically, we will introduce two different retractions with suitable local parametrizations. The first retraction is well known and quite straightforward. The second one uses a more sophisticated computation. We introduce it in order to have two different retractions to our disposal 
and to demonstrate that these two can be used interchangingly, as predicted by our theoretic results. 

\paragraph{Projection to the sphere.}
In the following we use the representation: 
\[
 T_{v}\mathbb{S}^2 = \{ w\in \R^3: w \perp v \}. 
\]
For $v\in \mathbb{S}^2$, let be $u\in \mathbb{R}^2$, $u=(u_1,u_2)$ and $\{\zeta_1,\zeta_2\}\in T_{v}\mathbb{S}^2$ be an orthogonal basis for the tangent space of $\mathbb{S}^2$ at every $v$.  We  define the parametrization around $v$ by:
 \begin{align*}
\mu_{v,p}(u) =\frac{v+u_1\zeta_1+u_2\zeta_2}{\Vert v+u_1\zeta_1+u_2\zeta_2 \Vert}
 \end{align*}
with first and second derivatives:
\begin{align*}
  \mu'_{v,p}(0)\delta u&=\delta u_1\zeta_1 + \delta u_2 \zeta_2\\
  \mu''_{v,p}(0)(\delta u,\delta w) &= -(\delta u_1\delta w_1+\delta u_2\delta w_2)v.
 \end{align*}
This parametrization implements the retraction:
 \begin{align*}
 R_{v,p}(\delta v)=\frac{v+\delta v}{\|v +\delta v \|}
 \end{align*}
and they satisfy $R_{v,p}(0)=\mu_{v,p}(0)=v$, $T_{0_v}R_{v,p} =id_{T_{v}\mathbb{S}^2}$.
Details can be found in \cite{absil2009optimization}.	
\paragraph{Matrix exponential.}
 The following alternative retraction uses a characterisation of $T_{v}\mathbb{S}^2$ via the space of skew-symmetric matrices $\mathfrak{so}(3)=\{H\in \mathbb{R}^{3\times 3}\vert H=-H^{T} \}$:
 \[
  T_{v}\mathbb{S}^2 = \{ H{v} : H \in \mathfrak{so}(3) \}.
 \]
 This follows from $\langle H{v},{v}\rangle=-\langle {v},H{v}\rangle=0$ by the fact that $Hv$ can be written as $w\times  v$, which is non-zero if $0\neq w \perp v$. 
Using the matrix exponential map, and setting $\delta v=H{v}$ we can define the following retraction: 
 \[
   R_{{v},e}(\delta v)=\exp(H){v}.
 \]
 where $\exp: \mathfrak{so}(3) \to SO(3)$
is the matrix exponential mapping with  
\[
SO(3)=\{Q\in \mathbb{R}^{3\times 3}\vert QQ^{T}=Q^{T}Q=id_{\R^3},\, \det(Q)=1 \},
\]
the group of rotations. This retraction is well defined: if $H_0v=0$, then $\exp(H_0)v=v$ as can be seen by the series expansion of the matrix exponential, and thus $\exp(H+H_0){v}=\exp(H)v$. 

For any given ${v} \in \mathbb{S}^2$ we consider a basis $b_2=\{C_1{v},C_2{v}\}$ for the tangent space $T_{v}\mathbb{S}^2$,
where $C_j \in \mathfrak{so}(3)$ are chosen in a way that $C_j {v}\neq 0$ for $j=1,2$. Now we define the map  for $u=(u_1,u_2)$:
 \begin{align*}
 {\mu}_{v,e}(u)=\exp \left(u_1 C_1 + u_2 C_2 \right)v
 \end{align*}
Since $\exp \left(u_1 C_1 + u_2 C_2 \right)\in SO(3)$ and $\exp(0)=I$ we obtain
${\mu}_{v,e}(0)=v$ and $\|{\mu}_{v,e}(u)\|_{\R^3}=1$, i.e., ${\mu}_{v,e}(u)\in \mathbb{S}^2$. 
Then first and second derivatives are:
\begin{align*}
  \mu_{v,e}'(0)\delta u&=(\delta u_1C_1+\delta u_2C_2)v\\
\mu''_{v,e}(0)(\delta u,\delta w)&=(\delta u_1 C_1 +\delta u_2 C_2)(\delta w_1 C_1 +\delta w_2 C_2)v .
\end{align*}
So the derivative of the retraction reads $T_{0_v}R_{v,e}=id_{T_{v}\mathbb{S}^2}$.


\subsection{The pullback of the problem}
We now pull back the energy functional $f$ and the constraint mapping $c$ using a local parametrization at each $v_i$ through $\mu_{v_i}$, which denotes any of the two parametrizations, presented above. From ($\ref{Functional}$) the pullback of the energy functional
takes the form:
\begin{align}\label{discretized_functional_charts}
\F(y,u)=\frac{\sigma}{2}\sum_{i=0}^{n-1} h\left\langle \frac{1}{h}(\mu_{v_{i+1}}(u_{i+1})- \mu_{v_i}(u_i)) ,\frac{1}{h}(\mu_{v_{i+1}}(u_{i+1})- \mu_{v_i}(u_i)) \right\rangle -\sum_{i=0}^{n-1} h \left\langle g_i,y_i \right\rangle.
\end{align}
(where $\mu_{v_{0}}(u_{0})$ and $\mu_{v_{n}}(u_{n})$ have to be replaced by the fixed $v_0$ and $v_n$ respectively) and componentwise:
\begin{align}\label{discretized_constrained_charts}
\C_i(y,u)=\frac{y_{i+1}-y_i}{h}-\mu_{v_i}(u_i)=0 \quad i=0\dots n-1.
\end{align}
We observe:
\[
 \F : (\R^3\times \R^2)^{n-1}\to \R, \mbox{ and } \C : (\R^3\times \R^2)^{n-1} \to (\R^3)^{n}.
\]
\paragraph{Derivatives of the pullbacks.}
Now we provide the derivatives of the involved pullbacks.  This is done through the composition with the local parametrizations of the sphere. The derivatives are computed centered at the zero of each 
tangent space parametrization of each sphere $\mathbb{S}^2$. For the composite step method, we need to compute first and second derivatives of both, energy and constraint mappings in charts.
Consider the discretized energy functional in \eqref{discretized_functional_charts}. Its first and second derivatives are given by:
\begin{align*}
\F\,'(y,u)(\delta y,\delta u)=\frac{\partial\F(y,u)}{\partial y}\delta y+\frac{\partial\F(y,u)}{\partial u}\delta u
\end{align*}
and
\begin{align*}
\F\,''(y,u)(\delta y,\delta u)^2=\left[\begin{array}{cc}
\delta y \\\delta u
\end{array}  \right]^T\left[\begin{array}{cc}
0 & 0  \\
0 & \frac{\partial^2\F(y,u)}{\partial u^2}
\end{array}\right]\left[\begin{array}{c}
\delta y \\
\delta u
\end{array}\right]
\end{align*} 
where 
\begin{align*}
\frac{\partial \F(y,u)}{\partial y_i}\delta y=-h\langle f_i,\delta y\rangle,
\end{align*} 
and, at $u=0$, taking into account that $\mu_{v_i}(0)=v_i$:
\begin{align*}
\frac{\partial \F(y,0)}{\partial u_i}\delta u
&=-\frac{1}{h}\left\langle \mu_{v_i}'(0)\delta u, v_{i+1}-2v_{i}+v_{i-1} \right\rangle 
\\
\frac{\partial^2 \F(y,0)}{\partial u_i^2}(\delta u,\delta w)&=-\frac{1}{h}\left\langle \mu''_{v_i}(0)(\delta u,\delta w), v_{i+1}-2v_{i}+v_{i-1} \right\rangle 
+\frac{2}{h}\left\langle \mu_{v_i}'(0)\delta u, \mu_{v_i}'(0) \delta w\right\rangle \\
\frac{\partial^2 \F(y,0)}{\partial u_i \partial u_{i-1} }(\delta u,\delta w)&=- \frac{1}{h}\left\langle\mu_{v_i}'(0)\delta u, \mu_{v_{i-1}}'(0)\delta w \right\rangle. 
\end{align*} 
The discretized constraint mapping  $\C$ in \eqref{discretized_constrained_charts} has the following derivatives:
\begin{align*}
\C_i'(y,0)(\delta y,\delta u)&=-\frac{1}{h}\delta y_i-\mu_{v_i}'(0)\delta u_i \\
\C_i''(y,0)(\delta y,\delta u)^2&=\left[\begin{array}{cc}
\delta y_i & \delta u_i
\end{array}  \right]\left[\begin{array}{cc}
0 & 0\\
0 & -\mu_{v_i}''(0)
\end{array} \right]\left[\begin{array}{c}
\delta y_i \\
\delta u_i
\end{array} \right].
\end{align*}

\subsection{Numerical Results}

We provide numerical simulations in order to illustrate the performance of the composite step method, described above. The optimization algorithm was implemented in Spacy \footnote{https://spacy-dev.github.io/Spacy/} which is a C++ library designed for optimization algorithms in a general setting and particularly suited for variational problems.

\begin{figure}[h!]%
    \centering
       \centering
    \includegraphics[width = 0.4\textwidth]{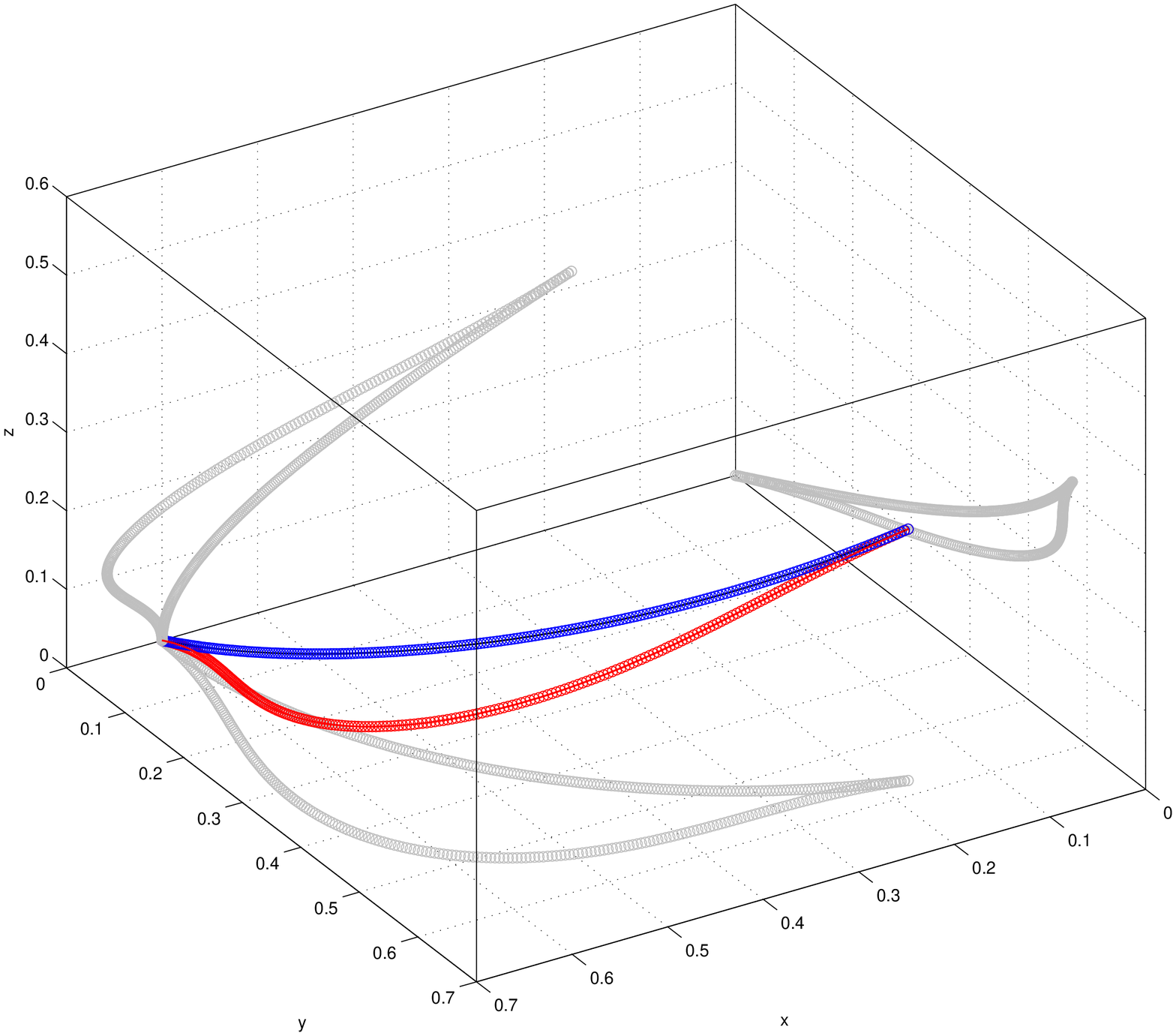}
    \hspace{1cm}
    \includegraphics[width = 0.4\textwidth]{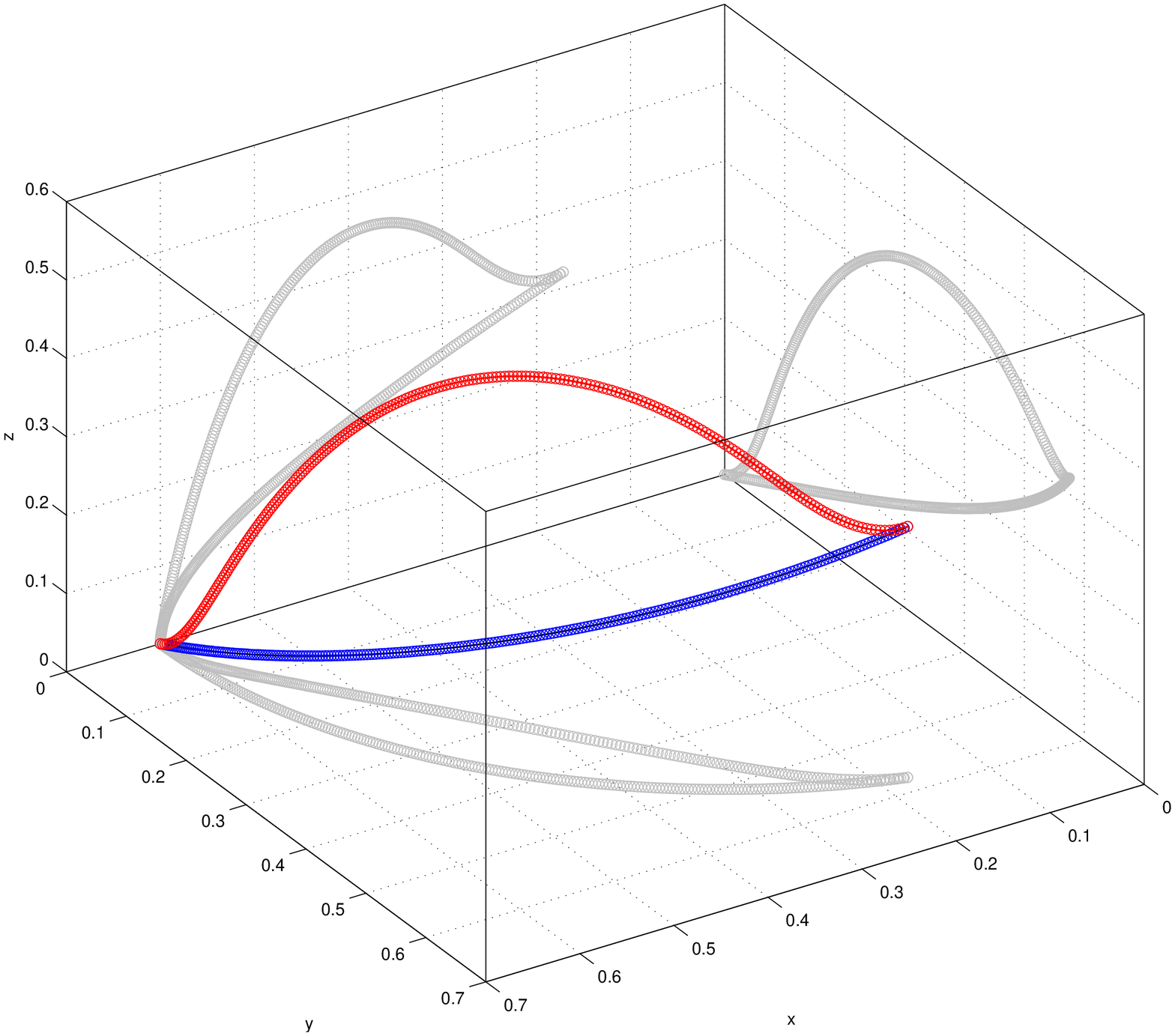}
    \caption{ Solutions of the rod problem, blue initial configuration, red computed solution, grey shades: projection to the coordinate planes. Left: rod without external force. Right: rod with external force $1000\,e_3$,
    (pointing upwards)}%
    \label{fig:example1}%
\end{figure}

We remind the problem setting:
\begin{align*}
\min_{(y,v)\in Y\times V} \frac{1}{2} \int_0^1 \sigma \left\langle v',v' \right\rangle \, ds-\int_0^1  \left\langle g, y \right\rangle  \, ds \hspace{0.5cm}s.t.\hspace{0.5cm}y'-v=0
\end{align*} 
where $\sigma>0$ is the stiffness of the rod, and $g$ describe the external loads. As initial configuration we consider a rod $(y_0,v_0)$ which assumes the form:
\begin{align*}
y_0(s)=[r\cos(\omega s),r\sin(\omega s),a^2\omega s], \qquad v_0(s)=y_0'(s)
\end{align*} 
with $s\in [0,1]$ $r>0$, $a>0$ and $\omega=\frac{1}{\sqrt{r^2+a^2}}$. The rod is clamped at $y_a=y_0(0)=[r,\, 0,\, 0]$ $y_b=y_0(1)=[r \cos(\omega), \, r\sin(\omega),\, a^2\omega]$ and $v_a=y_0'(0)/|y_0'(0)|$, $v_b=y_0'(1)/|y_0'(1)|$.  We perform 
numerical simulations for  $r=0.6$, $a=0.5$. The stiffness of the rod will be constant and given by $\sigma=1.0$. A minimization without external forces, using the exponential retraction $\mu_2$ and $n=240$ nodes converges in 7 iterations. The corresponding result can be seen in Figure~\ref{fig:example1}, left.

\begin{table}[h]
\begin{center}
 \begin{tabular}{|p{2.0cm}|p{2.0cm} p{2.0cm}| }
 
 \hline
 $\quad \Rd$ \textbackslash$\; \Rp$   & $R_{v,p}$  & $R_{v,e}$  \\
 \hline
   $R_{v,p}$   & 9  &  9   \\
   
   $R_{v,e}$  & 10  &  10 \\
  \hline
\end{tabular}
\end{center}

\vspace{0.1cm}
\caption{
Number of composite step iterations for different combinations of retraction. 
The pullback is done with the parametrization in the column and the update with the parametrization 
in the row. Here $R_{v,p}(\delta v)=\frac{v+\delta v}{\|v+\delta v\|}$ and $R_{v,e}(\delta v)=\exp(H)v$ with $\delta v=Hv$.  
}\label{tab:iter1}
\end{table}
\begin{table}[h]
\begin{center}
\begin{tabular}{|p{1.8cm}|p{1.8cm}| }
 \hline
 $n$ & \#iterations  \\
 \hline
    120  & 9   \\
    240  & 12   \\
    480  & 8   \\
    960  & 10   \\
   \hline
\end{tabular}
\end{center}
\caption{
Number of composite step iterations for the problem with different number of nodes $n$. The pullback and updates are done with the parametrization  $\mu_{v,e}(u)=\exp(C_1u_1+C_2u_2)v$. 
}\label{tab:iter2}
\end{table}

Next, we apply an external force $g=1000 e_3$ to the rod, where $e_3=[0,0,1]^T$ (cf. Figure~\ref{fig:example1}, right). 
We consider the two discussed retractions and combinations of them and observe  similar numbers of iterations in all cases (cf. Table~\ref{tab:iter1}). Also
the number of iterations is largely independent of the size of the grid (cf. Table~\ref{tab:iter2}).

\begin{figure}[h!]
   \centering
    \includegraphics[width = 0.4\textwidth]{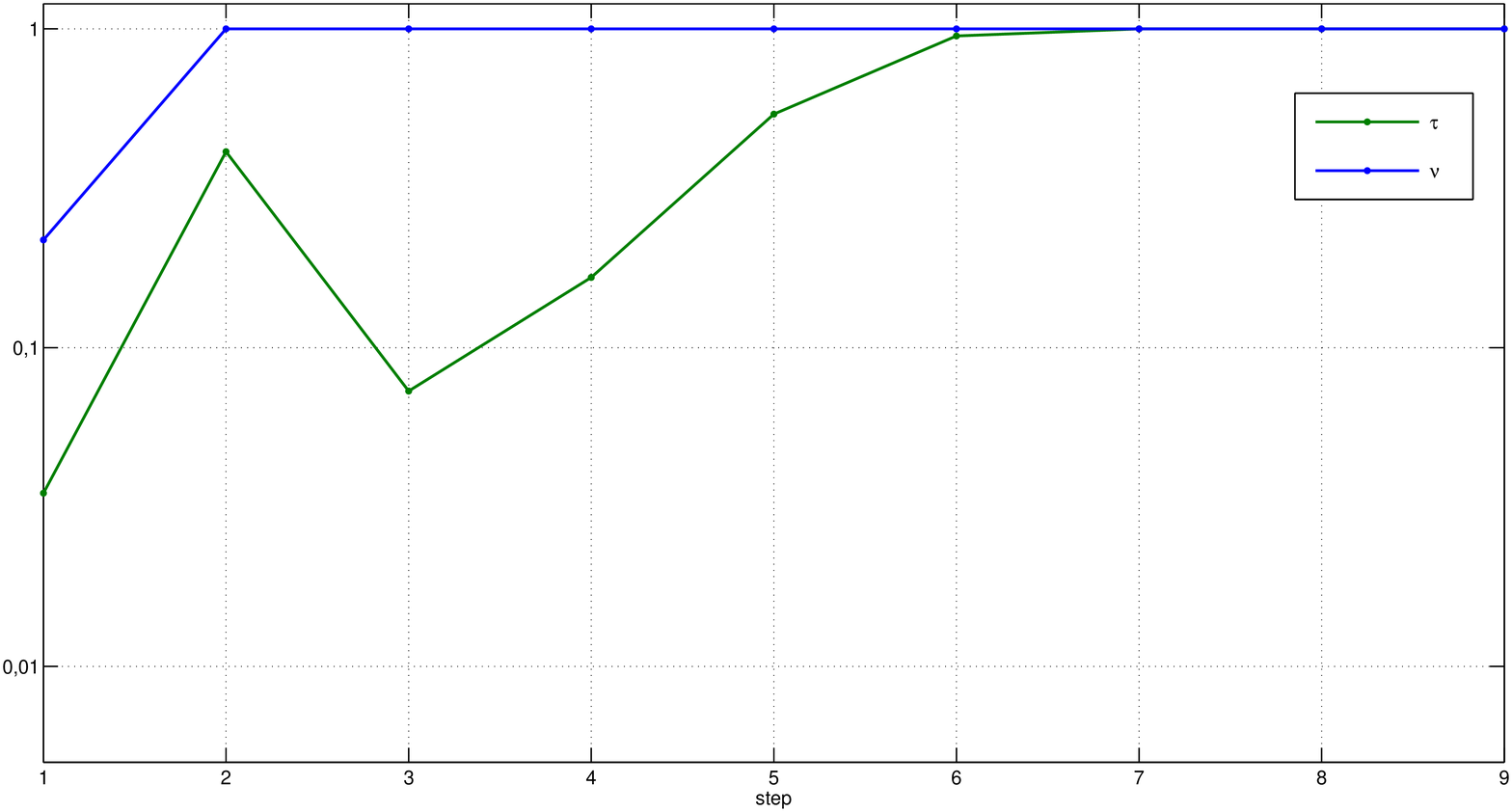}
    \hspace{1cm}
    \includegraphics[width = 0.4\textwidth]{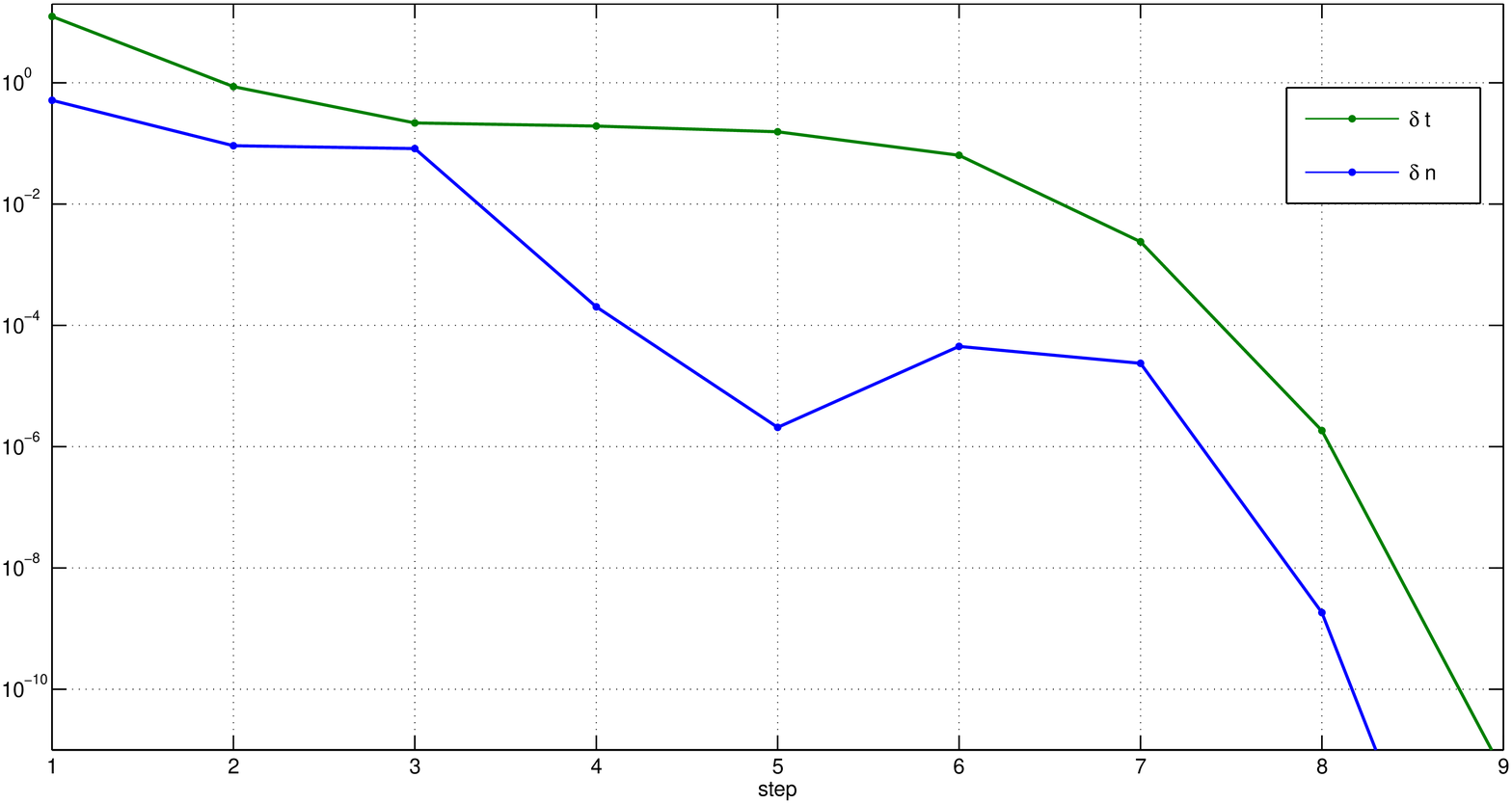}
  \caption{Iteration history: left: damping factors for normal and tangent steps, right: lengths of steps.
  }
  \label{fig:iteration}
\end{figure}

In Figure~\ref{fig:iteration} we take a closer look at the iteration history. We observe that after the globalization phase the damping factors are $1$ eventually and that the step sizes $\delta t$ and $\delta n$ become very small, close to the solution, indicating local superlinear convergence.

\section{Conclusion}

We have worked out in detail, how SQP methods can be extended to nonlinear manifolds by applying a two-step procedure: first the problem is pulled back to tangent spaces, using retractions and stratifications, second a linear-quadratic model is derived, whose minimization yields the required SQP steps. For this class of methods we derived results on fast local convergence, based on an affine covariant analysis. As a second theoretical issue, we studied the influence of retractions and stratifications on transition to fast local convergence of globalization methods. Here we could show for a specific globalization method that the Maratos effect can be avoided, if a modified quadratic model is used. 
Finally, we applied the described algorithm to a simple model problem from variational analysis: an inextensible elastic rod. The algorithm performed as expected and converged locally superlinearly.

Clearly, this work only scratches at the surface of the general topic of constrained optimization on manifolds and SQP methods in this setting. Given the variety of SQP algorithms that have been proposed on vector spaces, many alternatives to and extensions of the proposed method are conceivable. In particular, the extension of SQP methods to inequality constrained problems, for example by active-set or interior point methods, is an important open topic. Similarly, as already indicated in the introduction, there are plenty of applications to explore, which can be modelled as constrained optimization problems on manifolds. Also in this direction, a lot of work can be done in designing robust and efficient algorithms for their solution. In particular, variational problems in continuum mechanics call for the combination of optimization on manifolds and techniques of large scale numerical computations. 

\paragraph{Acknowledgement.}
This work was supported by the DFG grant SCHI 1379/3-1 ``Optimierung auf Mannigfaltigkeiten f\"ur die numerische L\"osung von gleichungsbesch\"ankten Variationsproblemen''

\bibliographystyle{alpha}
\bibliography{mybib}

\newcommand{\etalchar}[1]{$^{#1}$}
\begin{thebibliography}{KVBP{\etalchar{+}}14}

\bibitem[AKT12]{alouges2012convergent}
Fran{\c{c}}ois Alouges, Evaggelos Kritsikis, and Jean-Christophe Toussaint.
\newblock A convergent finite element approximation for
  {L}andau--{L}ifschitz--{G}ilbert equation.
\newblock {\em Physica B: Condensed Matter}, 407(9):1345--1349, 2012.

\bibitem[Alo97]{alouges1997new}
Fran{\c{c}}ois Alouges.
\newblock A new algorithm for computing liquid crystal stable configurations:
  the harmonic mapping case.
\newblock {\em SIAM journal on numerical analysis}, 34(5):1708--1726, 1997.

\bibitem[AMS08]{absil2009optimization}
Pierre-Antoine Absil, Robert Mahony, and Rodolphe Sepulchre.
\newblock {\em Optimization algorithms on matrix manifolds}.
\newblock Princeton University Press, 2008.

\bibitem[AR78]{antman1978global}
Stuart~S Antman and Gerald Rosenfeld.
\newblock Global behavior of buckled states of nonlinearly elastic rods.
\newblock {\em Siam Review}, 20(3):513--566, 1978.

\bibitem[Bal02]{ball2002some}
John~M Ball.
\newblock Some open problems in elasticity.
\newblock In {\em Geometry, mechanics, and dynamics}, pages 3--59. Springer,
  2002.

\bibitem[BEK18]{Brossette2018}
S.~{Brossette}, A.~{Escande}, and A.~{Kheddar}.
\newblock Multicontact postures computation on manifolds.
\newblock {\em IEEE Transactions on Robotics}, 34(5):1252--1265, 2018.

\bibitem[BH19]{HerzogBergmann2019}
Ronny Bergmann and Roland Herzog.
\newblock Intrinsic formulation of {KKT} conditions and constraint
  qualifications on smooth manifolds.
\newblock {\em SIAM J. Optim.}, 29(4):2423--2444, 2019.

\bibitem[BHM11]{bauer2010sobolev}
Martin Bauer, Philipp Harms, and Peter~W Michor.
\newblock Sobolev metrics on shape space of surfaces.
\newblock {\em Journal of Geometric Mechanics}, 3(1941 4889 2011 4 389):389,
  2011.

\bibitem[BP07]{bartels2007constraint}
S{\"o}ren Bartels and Andreas Prohl.
\newblock Constraint preserving implicit finite element discretization of
  harmonic map flow into spheres.
\newblock {\em Mathematics of Computation}, 76(260):1847--1859, 2007.

\bibitem[BS89]{badur1989influence}
J~Badur and Helmut Stumpf.
\newblock {\em On the influence of E. and F. Cosserat on modern continuum
  mechanics and field theory}.
\newblock Ruhr-Universit{\"a}t Bochum, Institut f{\"u}r Mechanik, 1989.

\bibitem[CGT00]{conn2000trust}
Andrew~R. Conn, Nicholas~I.M. Gould, and Philippe~L. Toint.
\newblock {\em Trust region methods}, volume~1.
\newblock Siam, 2000.

\bibitem[CGT11]{Cartis2011}
C.~Cartis, N.~Gould, and P.L. Toint.
\newblock Adaptive cubic regularisation methods for unconstrained optimization.
  {P}art {I}: motivation, convergence and numerical results.
\newblock {\em Math, Prog.}, 127(2):245--495, 2011.

\bibitem[Deu11]{deuflhard2011newton}
Peter Deuflhard.
\newblock {\em {N}ewton methods for nonlinear problems: affine invariance and
  adaptive algorithms}, volume~35.
\newblock Springer Science \& Business Media, 2011.

\bibitem[EL78]{eells1978report}
James Eells and Luc Lemaire.
\newblock A report on harmonic maps.
\newblock {\em Bulletin of the London mathematical society}, 10(1):1--68, 1978.

\bibitem[GLT89]{glowinski1989augmented}
Ronald Glowinski and Patrick Le~Tallec.
\newblock {\em Augmented Lagrangian and operator-splitting methods in nonlinear
  mechanics}.
\newblock SIAM, 1989.

\bibitem[HR14]{heinkenschloss2014matrix}
Matthias Heinkenschloss and Denis Ridzal.
\newblock A matrix-free trust-region {SQP} method for equality constrained
  optimization.
\newblock {\em SIAM Journal on Optimization}, 24(3):1507--1541, 2014.

\bibitem[HT04]{huper2004newton}
Knut Huper and Jochen Trumpf.
\newblock {N}ewton-like methods for numerical optimization on manifolds.
\newblock In {\em Signals, Systems and Computers, 2004. Conference Record of
  the Thirty-Eighth Asilomar Conference on}, volume~1, pages 136--139. IEEE,
  2004.

\bibitem[KVBP{\etalchar{+}}14]{kritsikis2014beyond}
Evaggelos Kritsikis, A~Vaysset, LD~Buda-Prejbeanu, Fran{\c{c}}ois Alouges, and
  J-C Toussaint.
\newblock Beyond first-order finite element schemes in micromagnetics.
\newblock {\em Journal of Computational Physics}, 256:357--366, 2014.

\bibitem[Lan01]{lang2001fundamentals}
S.~Lang.
\newblock {\em Fundamentals of Differential Geometry}.
\newblock Graduate Texts in Mathematics. Springer New York, 2001.

\bibitem[LB19]{BoumalLiu2019}
C.~Liu and N.~Boumal.
\newblock Simple algorithms for optimization on {R}iemannian manifolds with
  constraints.
\newblock {\em Applied Mathematics \& Optimiztion}, 2019.

\bibitem[LL89]{lin1989relaxation}
San-Yih Lin and Mitchell Luskin.
\newblock Relaxation methods for liquid crystal problems.
\newblock {\em SIAM Journal on Numerical Analysis}, 26(6):1310--1324, 1989.

\bibitem[LSW14]{lubkoll2014optimal}
Lars Lubkoll, Anton Schiela, and Martin Weiser.
\newblock An optimal control problem in polyconvex hyperelasticity.
\newblock {\em SIAM Journal on Control and Optimization}, 52(3):1403--1422,
  2014.

\bibitem[LSW17]{lubkoll2017affine}
Lars Lubkoll, Anton Schiela, and Martin Weiser.
\newblock An affine covariant composite step method for optimization with
  {PDE}s as equality constraints.
\newblock {\em Optimization Methods and Software}, 32(5):1132--1161, 2017.

\bibitem[LSZ11]{liu2011mathematical}
Wei Liu, Anuj Srivastava, and Jinfeng Zhang.
\newblock A mathematical framework for protein structure comparison.
\newblock {\em PLoS Computational Biology}, 7(2):e1001075, 2011.

\bibitem[Lue72]{luenberger1972gradient}
David~G Luenberger.
\newblock The gradient projection method along geodesics.
\newblock {\em Management Science}, 18(11):620--631, 1972.

\bibitem[Mar78]{maratos1978}
Nicolas Martos.
\newblock {\em Exact Penalty Function Algorithms for Finite Dimensional and
  Control Optimization Problems}.
\newblock PhD thesis, 1978.

\bibitem[Mie02]{mielke2002finite}
Alexander Mielke.
\newblock Finite elastoplasticity {L}ie groups and geodesics on sl (d).
\newblock In {\em Geometry, mechanics, and dynamics}, pages 61--90. Springer,
  2002.

\bibitem[MS04]{manton2004topological}
Nicholas Manton and Paul Sutcliffe.
\newblock {\em Topological solitons}.
\newblock Cambridge University Press, 2004.

\bibitem[NJ06]{Nocedal1999}
J.~Nocedal and Wright~S. J.
\newblock {\em Numerical Optimization}.
\newblock Springer, 2006.

\bibitem[Omo89]{Omojokun1990}
E.~O. Omojokun.
\newblock {\em Trust Region Algorithms for Optimization with Nonlinear Equality
  and Inequality Constraints}.
\newblock PhD thesis, Boulder, CO, USA, 1989.
\newblock UMI Order No: GAX89-23520.

\bibitem[PFA06]{pennec2006riemannian}
Xavier Pennec, Pierre Fillard, and Nicholas Ayache.
\newblock A {R}iemannian framework for tensor computing.
\newblock {\em International Journal of computer vision}, 66(1):41--66, 2006.

\bibitem[Pro95]{prost1995physics}
Jacques Prost.
\newblock {\em The physics of liquid crystals}, volume~83.
\newblock Oxford university press, 1995.

\bibitem[RW12]{ring2012optimization}
Wolfgang Ring and Benedikt Wirth.
\newblock Optimization methods on {R}iemannian manifolds and their application
  to shape space.
\newblock {\em SIAM Journal on Optimization}, 22(2):596--627, 2012.

\bibitem[Sch14]{schulz2014riemannian}
Volker Schulz.
\newblock A {R}iemannian view on shape optimization.
\newblock {\em Foundations of Computational Mathematics}, 14(3):483--501, 2014.

\bibitem[SS00]{shatah2000geometric}
Jalal M~Ihsan Shatah and Michael Struwe.
\newblock {\em Geometric wave equations}, volume~2.
\newblock American Mathematical Soc., 2000.

\bibitem[SSW15]{schulz2015towards}
Volker Schulz, Martin Siebenborn, and Kathrin Welker.
\newblock Towards a {L}agrange--{N}ewton approach for {PDE} constrained shape
  optimization.
\newblock In {\em New Trends in Shape Optimization}, pages 229--249. Springer,
  2015.

\bibitem[TSC00]{tang2000diffusion}
Bei Tang, Guillermo Sapiro, and Vicent Caselles.
\newblock Diffusion of general data on non-flat manifolds via harmonic maps
  theory: The direction diffusion case.
\newblock {\em International Journal of Computer Vision}, 36(2):149--161, 2000.

\bibitem[Var85]{Vardi1985}
A.~Vardi.
\newblock A trust region algorithm for equality constrained minimization:
  convergence properties and implementation.
\newblock {\em SIAM J. Numer. Anal.}, 22(3):575--591, 1985.

\bibitem[Zei86]{ZeidlerI}
E.~Zeidler.
\newblock {\em Nonlinear Functional Analysis and its Applications}, volume~I.
\newblock Springer, New York, 1986.

\end{thebibliography}

\end{document}